\documentclass{amsart}     
\usepackage{amsmath, amsfonts, amssymb}
\usepackage{eucal}
\usepackage{mathrsfs}
\usepackage{latexsym}
\usepackage{cite}
\usepackage{bbm}
\usepackage{mathtools}
\usepackage{graphicx}
\usepackage{relsize}
\usepackage{exscale}
\usepackage{cases}
\usepackage{comment}
\usepackage[usenames]{color}
\usepackage{xcolor}
\usepackage{cases}

\makeatletter
\renewcommand\section{\@startsection{section}{1}{\z@}{-3.5ex \@plus -1ex \@minus -.2ex}{2ex \@plus .2ex}
	{\centering\normalfont\large\scshape}}
\makeatother

\newcommand{\R}{\mathbb R}
\newcommand{\rbo}{\mathbb R}
\newcommand{\rn}{\mathbb R^n}
\newcommand{\sn}{S^{n-1}}

\newcommand{\kn}{\mathcal K^n}
\newcommand{\cn}{\mathcal C^n}
\newcommand{\kno}{\mathcal K_o^n}
\newcommand{\kne}{\mathcal K_e^n}
\newcommand{\csn}{C(\sn)}

\newcommand{\bx}{\pmb{x}}
\newcommand{\bu}{\pmb{\nu}}

\newcommand{\hm}{\mathcal H^{n-1}}
\newcommand{\hnn}{\mathcal H^n}
\newcommand{\hmi}{\mathcal H^i}
\newcommand{\Ln}{\mathscr{L}^n}

\newcommand\wtilde[1]{\overset{\lower.4ex\hbox{$\scriptstyle \sim$}}{#1}}
\newcommand\wst[1]{\overset{\lower.5ex\hbox{$\scriptscriptstyle \sim$}}{#1}}
\newcommand{\blb}{\raise.3ex\hbox{$\scriptstyle \pmb \lbrack$}}
\newcommand{\sblb}{\raise.1ex\hbox{$\scriptscriptstyle \pmb \lbrack$}}
\newcommand{\brb}{\raise.3ex\hbox{$\scriptstyle \pmb \rbrack$}}
\newcommand{\sbrb}{\raise.1ex\hbox{$\scriptscriptstyle \pmb \rbrack$}}
\newcommand{\bla}{\raise.2ex\hbox{$\scriptstyle\pmb \langle$}}
\newcommand{\sbla}{\raise.1ex\hbox{$\scriptscriptstyle\pmb \langle$}}
\newcommand{\bra}{\raise.2ex\hbox{$\scriptstyle\pmb \rangle$}}
\newcommand{\sbra}{\raise.1ex\hbox{$\scriptscriptstyle\pmb \rangle$}}
\newcommand{\blrb}{\raise.3ex\hbox{$\scriptstyle \pmb | $}}
\newcommand{\sblrb}{\raise.1ex\hbox{$\scriptscriptstyle \pmb | $}}

\newcommand{\st}{\scriptstyle}

\newcommand{\wt}{\widetilde}

\newcommand{\psum}{\,{+_{\negthinspace\kern-2pt p}}\,}
\newcommand{\qsum}[1]{\,{+_{\negthinspace\kern-2pt \lower -2pt \hbox{$_{_{#1}}$}}}\,}
\newcommand{\osum}{{+_{\negthinspace\kern-2pt {\rm{o}}}}\,}
\newcommand{\dpsum}{\,{\tilde+_{\negthinspace\kern-1pt p}}\,}
\newcommand{\dqsum}[1]{{\,\wt+_{\negthinspace\kern-1pt #1}}\,}
\newcommand{\lsub}[1]{\hskip -1.5pt\lower.5ex\hbox{$_{#1}$}}
\newcommand{\llsub}[1]{\hskip -1.5pt \raisebox{-.5pt}{$_{_{#1}}$}}
\newcommand{\rhok}{\rho\lsub{K}}

\newcommand{\di}{\vspace{3pt}\noindent\raisebox{2pt}{$\st\bullet$\ }}

\newcommand{\dr}[1]{{\rm ({\bf #1})}}

\newcommand{\thh}{\mathcal{H}}

\newcommand{\vv}{\mathrm{v}}

\newcommand{\eop}{\operatorname{E}}
\newcommand{\interior}{\operatorname{int}}

\DeclareMathOperator{\Int}{int}
\DeclareMathOperator{\Span}{span}
\DeclareMathOperator*\lowlim{\underline{lim}}

\begin{document}                        


\title{ Chord Measures in Integral Geometry \\ and Their Minkowski Problems }

\author{Erwin Lutwak}
\address{NYU Courant Institute of Mathematical Sciences}
\author{Dongmeng Xi}
\address{Shanghai University \&	Courant Institute}
\author{Deane Yang}
\address{NYU Courant Institute of Mathematical Sciences}
\author{Gaoyong Zhang}
\address{NYU Courant Institute of Mathematical Sciences}





\begin{abstract}
	
	To the families of geometric measures of convex bodies (the area measures of
	Aleksandrov-Fenchel-Jessen, the curvature measures of Federer, and the recently discovered dual curvature measures)
	a new family is added. The new family of geometric measures, called chord measures,
	arises from the study of integral geometric invariants of convex bodies.
	The Minkowski problems for the new measures and their logarithmic variants are proposed and attacked. When the given `data' is sufficiently regular, these problems
	are a new type of fully nonlinear partial differential equations involving dual quermassintegrals of functions.
	Major cases of these Minkowski problems are solved without regularity assumptions.
\end{abstract}

\maketitle

\numberwithin{equation}{section}

\newtheorem{theo}{Theorem}[section]
\newtheorem{coro}[theo]{Corollary}
\newtheorem{lemm}[theo]{Lemma}
\newtheorem{claim}[theo]{Claim}
\newtheorem{defi}[theo]{Definition}
\newtheorem{prop}[theo]{Proposition}
\newtheorem{conj}[theo]{Conjecture}
\newtheorem{exam}[theo]{Example}
\newtheorem{prob}[theo]{Problem}
\newtheorem{rema}[theo]{Remark}

\theoremstyle{definition}
\newtheorem*{definition-non}{Definition}

\newcommand\blfootnote[1]{%
	\begingroup
	\renewcommand\thefootnote{}\footnote{#1}%
	\addtocounter{footnote}{-1}%
	\endgroup
}




\section{  Introduction}

Convex geometric analysis as the analytic part of convex geometry is focused on the study of
geometric invariants and geometric measures associated with convex bodies in Euclidean space.
For general convex bodies, geometric measures play the roles that curvatures do for smooth hypersurfaces.
Geometric measures and geometric invariants are closely related.
When geometric invariants are viewed as
geometric functionals of convex bodies, geometric measures
often arise as differentials of the geometric invariants.
In the classical Brunn-Minkowski theory in convex geometric analysis,
{\it quermassintegrals} (such as, e.g., volume, surface area, and mean width) are the fundamental
geometric invariants. {\it Area measures}
introduced by Aleksandrov, Fenchel \& Jessen and {\it curvature measures} introduced by Federer
are the two major families of geometric measures of convex bodies
within the classical Brunn-Minkowski theory. Both families of
geometric measures are associated with quermassintegrals.
Quermassintegrals, area measures and their extensions (i.e., mixed volumes and mixed
area measures) are the main objects studied
in the classical Brunn-Minkowski theory developed by Minkowski, Aleksandrov, Fenchel, Blaschke, et al.
in the earlier part of the twentieth century. See Schneider's classic volume \cite{S14}.

In the 1970s, the dual Brunn-Minkowski theory was introduced in \cite{L75paci}
from a duality point of view.
It demonstrates a highly nontrivial duality in convex geometry, and also shows connections with
dualities in other subjects such as functional analysis. The fundamental geometric invariants
in the dual
Brunn-Minkowski theory are called {\it dual quermassintegrals}. Studying dual concepts (in this
geometric duality) to the concepts in the classical Brunn-Minkowski theory has led to
deep results.
An example is the Busemann-Petty problem in the dual Brunn-Minkowski theory,
which is much harder and deeper than its counterpart in the classical Brunn-Minkowski theory.
(See for example, \cite{Bo91, G94annals, K03, K04, GKS99, L88adv, Z99annals}, and see the books
Gardner \cite{G06book} and Koldobsky \cite{K05} for additional references.)
The geometric measures associated with dual quermassintegrals, called
{\it dual curvature measures},
were introduced very recently in \cite{HLYZ}, a work which presents the long-elusive duals of Federer's curvature measures
within the dual Brunn-Minkowski theory.

Around the mid-1930s, Blaschke established a school of Integral Geometry
in Hamburg. His school began a systematic study of convex geometry from a probabilistic point of view,
under the name of integral geometry. (See Santal\'o \cite{San} and Ren \cite{Ren}.)
Various interesting integral formulas for and involving
geometric invariants of convex bodies were established. Many of these later become important tools
in stochastic geometry. (See
Schneider-Weil \cite{SW}.) An important direction of study in integral geometry is establishing similar
integral formulas for non-convex objects such as manifolds. This started with Chern's earlier work
on the fundamental kinematic formula from a differential geometric viewpoint (see \cite{San})
and has seen dramatic developments in e.g.
\cite{Ales1, Ales2, BF11annals, BFS14gafa, HP14jams, Lud03, LR10annals}.
The study we now present is from a new perspective. Here, we open a new direction of study within integral geometry:
one aimed toward solving geometric equations which, when the given \lq data\rq\ is sufficiently regular,
are fully equivalent to partial differential equations involving dual quermassintegrals  of functions (see \eqref{0.9.991} and \eqref{0.9.992} for definition).

While in integral geometry, quermassintegrals of convex bodies remain the basic geometric invariants,
new fundamental geometric invariants arise.
{\it Chord (power) integrals} are such integral geometric invariants.
They are moment-integrals of random lines intersecting a convex body. They enjoy several important properties:
in convex geometric analysis, chord-integrals are translation invariant
analogues of the origin-dependent dual quermassintegrals \cite{Z99tams};
in integral geometry, the second moment of the area of a random plane intersecting
a convex body can be expressed as a chord-integral \cite{Z99tams}; in geometric tomography,
chord-integrals are closely related to $X$-rays \cite{G06book}; and in analysis,
integrals of Riesz potentials of
characteristic functions of convex bodies can be written as chord-integrals \cite{Ren, San}.

Chord-integrals form another main family of geometric invariants of convex bodies that stand alongside the
quermassintegrals and the dual quermassintegrals.
One of the purposes of this work is to introduce the geometric measures derived from
chord-integrals. We call them {\it chord measures}. This family of
geometric measures of convex bodies will join Aleksandrov-Fenchel-Jessen's area measures,
Federer's curvature measures, and the recently discovered dual curvature measures.
Once chord measures are constructed, their $L_p$ extensions follow naturally.
The critical $L_0$ case, which we call {\it cone-chord measures}, will then be studied.
We shall then tackle the Minkowski problem for chord measures as well as the log-Minkowski problem
for cone-chord measures. The log-Minkowski problem for cone-chord measures is intimately
connected with the concentration problem for cone-chord measure.

This work establishes new connections between
integral geometry and geometric measure theory and demonstrates how partial differential
equations arise naturally within
integral geometry. All this opens new directions of investigation within integral geometry,
different from classical integral geometry. In the immediately following section we make the
above more precise beginning with a quick review of the classical construction of
geometric measures of convex bodies.

\vspace{5pt}

\noindent
{\bf Construction in the Classical Brunn-Minkowski Theory.}
The fundamental {\it surface area measure} $S(K,\cdot)$ associated with the convex body $K\subset\rn$ is
a Borel measure on the unit sphere $\sn$ which (as will be seen) can be directly  derived from $V$,
the volume functional (i.e. Lebesgue measure in $\rn$). If $K$ is a sufficiently smooth
body, then $S(K,\cdot)$ has a density which is the reciprocal Gauss curvature of $\partial K$
viewed as a function of outer unit normals of $\partial K$.
The geometric definition of $S(K,\cdot)$ for an arbitrary convex body $K\subset\rn$ is
\begin{equation}\label{0.0}
	S(K, \eta) = \hm(\nu_K^{-1}(\eta)), \ \ \ \ \ \  \text{Borel}\ \eta \subset \sn,
\end{equation}
where $\hmi$ is $i$-dimensional
Hausdorff measure and
$\nu_K:\partial K \to\sn$ is the Gauss map.\negthinspace\negthinspace\footnote{The Gauss map is
	well known to be defined
	$\hm$-almost everywhere on $\partial K$.} There is an analytic construction of the surface area measure
which shows how the surface area measure $S(K,\cdot)$ arises from the derivative of
the volume functional $V$ when evaluated at $K$.
For all convex bodies $K,L\subset\rn$,
\begin{equation}\label{0.1}
	\frac{d}{dt}\Big|_{t=0^+} V(K+tL) = \int_{\sn} h_L(v)\, dS(K,v),
\end{equation}
where $h_Q(v)=\max\{v\cdot x : x\in Q\}$ denotes the support function of the compact,
convex $Q\subset\rn$, while $v\cdot u$ is the inner product of $u$ and $v$ in $\rn$,
and where the linear combination $K+tL$, with $t>0$, is the convex body that's just
the usual vector sum, or equivalently whose support function is given by $h_{K+tL}=h_K+th_L$.
Thus the surface area measure may be viewed as the linearization of the volume functional.

Another fundamental geometric measure that has a clear geometric meaning is
{\it cone-volume measure} $V(K,\cdot)$ of a convex body $K$ that contains
the origin. It is a Borel measure on $\sn$ defined by
\begin{equation}\label{0.01}
	V(K,\eta) = \mathcal{H}^n(K\cap c_K(\eta)), \ \ \ \ \ \  \text{Borel}\ \eta \subset \sn,
\end{equation}
where
$c_K(\eta)$ is the cone of rays emanating from the
origin such that $\partial K \cap c_K(\eta) = \nu_K^{-1}(\eta)$.
Cone-volume measure can also be defined by using the support function and the surface area measure:
\begin{equation}\label{0.2}
	dV(K,\cdot) = \frac1n h_K \, dS(K,\cdot),
\end{equation}
a formula which resembles the elementary volume formula for a cone.

The total measure of cone-volume measure is volume
and the total measure of the surface area measure
is surface area; i.e.,
\begin{equation}\label{0.3}
	V(K,\sn) = V(K), \hspace{60pt} S(K,\sn) = S(K),
\end{equation}
where $S(K)=\hm(\partial K)$ denotes the surface area of $K$.

In the classical Brunn-Minkowski theory, the construction \eqref{0.1} is extended to
all quermassintegrals. After volume, the next important quermassintegral is surface area,
and we are led to the fact that for all convex bodies $K,L\subset\rn$,
\begin{equation}\label{0.4}
	\frac{d}{dt}\Big|_{t=0^+} S(K+tL) = (n-1) \int_{\sn} h_L(v)\, dS'(K,v),
\end{equation}
where $S'(K,\cdot)$ is a Borel measure on the unit sphere.\negthinspace\footnote{In the literature, the surface area measure
	$S(K,\cdot)$ is often called the  $(n-1)$-th area measure of $K$ and
	is written as $S_{n-1}(K,\cdot)$, while $S'(K,\cdot)$ is often denoted by $S_{n-2}(K,\cdot)$ and is
	called the $(n-2)$-th area measure of $K$.}
When the boundary $\partial K$
is $C^2$ and of positive curvature,
\begin{equation}\label{0.4.0}
	S'(K, \eta)  = \int_{\nu_K^{-1}(\eta)} H(z) \, d\hm(z),  \ \ \ \text{Borel}\  \eta\subset \sn,
\end{equation}
where $H(z)$ is the mean curvature of $\partial K$ at $z\in \partial K$.

The above geometric measures and their constructions are
fundamental in the Brunn-Minkowski theory of convex bodies. Their generalizations
to all the quermassintegrals
(other than volume and surface area)
were developed by Aleksandrov and Fenchel \& Jessen in the 1930s.
In the 1990s, in his landmark paper \cite{Je96}, Jerison gave a harmonic analysis construction for
the electrostatic (or Newtonian) capacity and solved the associated Minkowski problem.
Recently,
this basic construction was finally successfully mirrored for the dual quermassintegrals
within the dual Brunn-Minkowski theory \cite{HLYZ}
where the long-sought dual curvature measures were finally uncovered.

\vspace{5pt}

\noindent
{\bf Construction in the Dual Brunn-Minkowski Theory. }
For a convex body $K\subset\rn$, an interior point $z\in K$, and $q\in\R$,
the $q$-th dual quermassintegral $\wt V_q(K, z)$ of $K$ with respect to $z$ is
\begin{equation}\label{0.9.991}
	\wt V_q(K, z) = \frac1n \int_{\sn} \rho\lsub{K,z}(u)^q\, du,
\end{equation}
where $\rho\lsub{K,z}(u) = \max\{\lambda \ge 0 : z+\lambda u\in K\}$ is the radial
function of $K$ with respect to $z$, and $du$ denotes spherical Lebesgue integration. We will need to extend definition \eqref{0.9.991} for the case where $z\in\partial K$ when $q>-1$.
Note that $\wt V_n(K, z) =V(K)$, the volume of $K$, for each $z$.
For a continuous function $h:\sn\to (0,\infty)$, the {\it Wulff shape} of $h$ is just the convex body
\begin{equation}\label{0.9.992}
	[h]=\{x\in\rn: x\cdot u \le h(u)\ \text{for all $u\in\sn$}\},
\end{equation}
and $\wt V_q(h, z)$ will be written to indicate $\wt V_q([h], z)$, when $z$ is such that
the inner product $z\cdot u < h(u)$ for all $u\in\sn$.

Dual quermassintegrals are basic geometric invariants of a convex body,
which also have an analytical interpretation. Indeed, for $q>0$, there is the formula,
\begin{equation}\label{0.9.90}
	\wt V_q(K, z) = \frac qn \int_{\rn} \frac{{\pmb 1}_K(x)}{|x-z|^{n-q}}\, dx, \ \ \ \text{real $q> 0$,}
\end{equation}
which is the Riesz potential of the characteristic function of the convex body $K$.
A surprising (to the authors) connection between dual quermassintegrals and mean curvature will be described in the next subsection.

Calculating the variation of the dual quermassintegrals of an arbitrary convex body is not easily accomplished and was only completed recently in \cite{HLYZ}. The variation leads to the $q$-th dual curvature measure
$\wt C_q(K,z;\cdot)$, and can be defined as the unique Borel measure on $\sn$ that satisfies for all convex bodies, $L$,
\[
\frac{d}{dt}\Big|_{t=0^+} \wt V_q(K+tL,z)
= q \int_{\sn} \frac{h_{L,z}(v)}{h_{K,z}(v)}\, d\wt C_q(K,z;v), \ \ \  \text{real $q\neq 0$,}
\]
where $h_{Q,z}(v) = \max\{v \cdot (y-z) : y\in Q\}$
is the support function (at $v\in\sn$) of the convex body $Q$ with respect to $z\in Q$.

In the amazing duality in convex geometry between the Brunn-Minkowski theory and the dual Brunn-Minkowski theory
dual curvature measures are dual to Federer's curvature measures for convex bodies.
Cone volume measure $V(K,\cdot)$ is the dual curvature measure $\wt C_n(K,0,\cdot)$
with respect to the origin. Thus, dual curvature measures can also be viewed as a family
of geometric measures that generalizes cone volume measure. While cone volume measure is origin
dependent, surface area measure is translation invariant; i.e., $S(x+K,\cdot)=S(K,\cdot)$, for all $x\in\rn$.
By using dual quermassintegrals, we will construct a family of
translation invariant geometric measures $F_q(K,\cdot)$, called {\it chord measures},
which, among others, include both
the surface area measure $S(K,\cdot)$ and the area measure $S'(K,\cdot)$ as special
cases. Chord measures are naturally associated with translation invariant integral geometric invariants.

\vspace{5 pt}

\noindent
{\bf Construction in Integral Geometry. }
Let $G_{n,i}$ and $G_{n,i}^a$ denote respectively the Grassmannian of $i$-dimensional
subspaces and  $i$-dimensional affine subspaces of $\rn$. For $G_{n,1}^a$ we will write
$\mathscr{L}^n$. We assume that all of these are endowed with their standard invariant Haar measures.

For each convex body $K$, the set of all lines in  $\mathscr{L}^n$ that meet $K$ has
positive but finite Haar measure in $\mathscr{L}^n$ and can thus be normalized so this
set of lines becomes a probability space.   The chord integrals of $K$ are essentially
the moments of the random variable $|K\cap\cdot|$ on this probability space.

Elements of $\mathscr{L}^n$
(i.e., lines) are determined
by rotations and parallel translations. Haar measure on $\mathscr{L}^n$ will be normalized
to be a probability measure when restricted to rotations and to be $(n-1)$-dimensional
Lebesgue measure when restricted to parallel translations.

The {\it chord integral} $I_q(K)$ of the convex body $K\subset\rn$ is defined by
\begin{equation}\label{0.5}
	I_q(K) = \int_{\mathscr{L}^n}
	|K\cap \ell|^q\, d\ell, \ \ \ \ \ \text{real $q\ge  0$},
\end{equation}
where $|K\cap \ell|$ denotes the length of the chord $K\cap \ell$, and where the integration
is with respect to Haar measure on $\mathscr{L}^n$ as normalized above.
Chord integrals are basic integral-geometric invariants which arise in integral geometry.
As can be seen in \eqref{0.6}, the formulas below show that the chord integrals of an arbitrary
convex body $K\subset\rn$
include the body's volume and surface area as two important special cases:
\begin{equation}\label{0.6}
	I_1(K) =  V(K), \ \
	I_0(K)= \frac{\omega_{n-1}}{n\omega_n} S(K), \ \
	I_{n+1}(K) = \frac{n+1} {\omega_n} V(K)^2,
\end{equation}
where $\omega_n$ is the $n$-dimensional volume of the unit ball, $B^n$, in $\rn$.
These are Crofton's volume formula, Cauchy's integral formula for surface area, and
the Poincar\'e-Hadwiger formula, respectively (see \cite{Ren, San}).

For some integer values of the index $q$, chord integrals are not only integral geometric
invariants of lines, but also integral geometric invariants of planes; specifically,
\begin{equation}\label{0.7}
	\int_{G_{n,i}^a} \vv_i(K\cap \xi_i)^2\, d\xi_i = \frac{\omega_i}{i+1} I_{i+1}(K), \ \ \
	\int_{G_{n,i}^a} \vv_i(K\cap \xi_i)\, d\xi_i = I_1(K),
\end{equation}
(see \cite{Z99tams}).
Here $\vv_i$ is $i$-dimensional volume and the normalization of the Haar measure on $G_{n,i}^a$ is easily discerned from \eqref{0.7}.

A relationship between chord integrals and dual quermassintegrals is given
by,
\begin{equation}\label{0.81}
	I_q(K) = \frac{q}{\omega_n} \int_K \wt V_{q-1}(K,z) \, dz, \ \ \  \text{real $q> 0$},
\end{equation}
(see \cite{GZ98, Z99tams}). In light of \eqref{0.81} and \eqref{0.9.90}, we see that for $q>1$
the relationship between chord integrals and the Riesz potentials of
characteristic functions of convex bodies is given by
\begin{equation}\label{0.8}
	I_q(K) = \frac{q(q-1)}{n\omega_n} \int_{\rn}\int_{\rn}
	\frac{{\pmb 1}_K(x) {\pmb 1}_K(y)}{|x-y|^{n+1-q}}\,
	dxdy, \ \ \ \ \text{real $q> 1$,},
\end{equation}
(see \cite{Ren, San}).

One of the aims of this paper is to do for the functionals $I_q$ what was done for the volume functional $V$ in \eqref{0.0}-\eqref{0.4.0}, and the dual mixed volume functionals  $\wt V_q$ in \cite{HLYZ}.
For each real $q>0$, we shall define the geometric measures $F_q(K,\cdot)$, for each convex body $K$ in $\rn$, by
\begin{equation}\label{0.12}
	F_q(K,\eta) = \frac{2q}{\omega_n} \int_{\nu_K^{-1}(\eta)} \wt V_{q-1}(K, z)\, d\hm(z),
	\ \  \text{Borel $\eta\subset \sn$.}
\end{equation}
It will be shown that each $F_q(K,\cdot)$ is a finite Borel measure on $\sn$ that is the differential of $I_q$ at $K$; i.e.,
\begin{equation}\label{0.9}
	\frac{d}{dt}\Big|_{t=0^+} I_q(K+tL) = \int_{\sn} h_L(v)\, dF_q(K,v),
\end{equation}
for each convex body $L$.
We call the measure $F_q(K,\cdot)$ the $q$-th {\it chord measure} generated by $K$. These measures are invariant under translations of the body $K$.

For $q=1$ the measure $F_q$ is the classical surface area measure $S$; i.e., for each convex body $K\subset\rn$,
\[
F_1(K,\cdot) = S(K, \cdot).
\]
The extremal case $q\to 0$ is very interesting. A surprising connection between dual
quermassintegrals and mean curvature will be presented: For each smooth convex body $K$
and each $z\in \partial K$, it will be shown that
\[
\lim_{q \to 0} q \wt V_{q-1}(K,z) = \frac{(n-1)\omega_{n-1}}{2n} H(z),
\]
where $H(z)$ is the mean curvature of $\partial K$ at $z$. This, \eqref{0.12} and \eqref{0.4.0}
imply that for
a smooth convex body
the area measure $S'(K,\cdot)$ is a limiting case of chord measures,
\begin{equation}\label{0.13}
	\lim_{q\to 0^+} F_q(K, \eta) = \frac{(n-1)\omega_{n-1}}{n\omega_n} S'(K, \eta),
\end{equation}
for each Borel set $\eta \subset \sn$.
Thus, we shall define
\[
F_0(K,\cdot)= \frac{(n-1)\omega_{n-1}}{n\omega_n} S'(K, \cdot).
\]
However, it is not clear if for each convex body $K$ and Borel set $\eta \subset \sn$,
the function $q\mapsto F_q(K, \eta)$ is continuous at $0$.

For real $q\ge0$ and each convex body $K$ containing the origin (but not necessarily in its interior) the
$q$-th {\it cone-chord measure} generated by $K$ will be defined by:

\begin{equation}\label{0.10}
	dG_q(K,\cdot) =\frac1{n+q-1} h_K\, dF_q(K,\cdot).
\end{equation}
Note that $G_1(K,\cdot)$ is the cone volume measure $V(K,\cdot)$.
It will be shown that
\begin{equation}\label{0.11}
	I_q(K)= G_q(K,\sn)
	.
\end{equation}

\vspace{5pt}

\noindent
{\bf Minkowski Problems.} Once a new geometric measure associated with convex bodies is discovered,
its Minkowski problem becomes the immediate challenge:
given a measure, is it possible to construct the body that generates that given measure?
This has become a core area within
convex geometric analysis, a subject which is parallel to the study of isoperimetric
problems for geometric invariants. Both areas have many open problems to be tackled.

The problem of constructing a convex body by prescribing its surface area measure is
the {\it classical Minkowski problem}. Minkowski posed and studied the problem and was able
to fully solve the polytope case.
The general case was solved independently by Aleksandrov and Fenchel \& Jessen (see \cite{S14}).
The regularity of solutions to the problem has had a huge
impact on fully nonlinear partial differential equations. See for example Trudinger-Wang \cite{TW}.

The Minkowski problem for the area measure $S'(K,\cdot)$ is one of the
{\it Christoffel-Minkowski
	problems}. All but the simplest of these remain long standing, open and important problems
(see Guan-Ma \cite{GM03} for a regular case).

The Minkowski problem for cone-volume
measure is called the {\it logarithmic Minkowski problem}.
It was solved in \cite{BLYZ13jams} for symmetric convex bodies. The solution depends on the manner
in which cone volume measure can concentrate on subspaces.
The general case remains open. The uniqueness problem for cone-volume measure
is also open. The solution is closely related to the conjectured {\it log-Brunn-Minkowski inequality},
a conjectured inequality that
has attracted much attention (see e.g., \cite{BLYZ12adv, CHLL, KL, KM, Sar}).
Its connections with Gauss curvature
flows were studied in e.g.\ \cite{And99, And03, LSW, CHZ}.

The Minkowski problem for the dual curvature measures is called the {\it dual Minkowski problem}.
It includes the log-Minkowski problem as a special case.
See e.g.,\ \cite{BLYZZ1, HLYZ, HZ, LSW, Zhao17cvpde, Zhao17jdg, GHXY1, GHXY2} for recent work.

The Minkowski problem of prescribing chord measures
is:

\vspace{5pt}

\noindent
{\bf The Chord Minkowski Problem.} {\it
	Suppose real $q\ge0$ is fixed. If $\mu$ is a finite Borel measure on $\sn$,
	what are necessary and sufficient conditions on $\mu$ to guarantee the existence of a convex body
	$K\subset\rn$ that solves the equation,
	\begin{equation}\label{0.14}
		F_q(K,\cdot) = \mu\, ?
	\end{equation}
}

Two classical Minkowski problems are included as special cases. The case of $q=1$
is the classical Minkowski problem for surface area measure which was solved by
Aleksandrov and Fenchel \& Jessen, and the case of $q=0$ is
the long standing unsolved (when $n>3$) Christoffel-Minkowski problem for the area measure $S'$.

When the given measure $\mu$ has a density $\frac{\omega_n}{2q}f$ that is an integrable nonnegative function on $\sn$,
equation \eqref{0.14} becomes a new type of Monge-Amp\`ere type partial differential equation,
\begin{equation}\label{0.15}
	\det\big(\nabla_{ij}h + h\delta_{ij}) = \frac{f}{\wt V_{q-1}(h,\nabla h)}, \ \ \ \text{on $\sn$},
\end{equation}
where $h$ is the unknown function on $\sn$ which is extended via homogeneity to $\rn$, while $\nabla h$
is the Euclidean gradient of $h$ in $\rn$, the spherical Hessian of $h$ with respect to an orthonormal frame on $\sn$ is $(\nabla_{ij} h)$, and $\delta_{ij}$ is the Kronecker delta.
Note that as opposed to the usual Monge-Amp\`ere type equations, the geometric functional $\wt V_{q-1}(h, \nabla h)$ can not be written as a composition of $h$ and $\nabla h$ with a multivariable function $F(t, z)$, with $(t,z)\in \R\negthinspace\times\negthinspace\rn$.
\vspace{5pt}

The Minkowski problem prescribing cone-chord measures is:

\vspace{5 pt}

\noindent
{\bf The Chord Log-Minkowski Problem.} {\it Suppose real $q\ge0$ is fixed. If $\mu$ is a finite Borel measure on $\sn$,
	what are necessary and sufficient conditions on $\mu$ to guarantee the existence of a convex body
	$K\subset\rn$ containing the origin that solves the equation,
	\begin{equation}\label{0.16}
		G_q(K,\cdot) = \mu\, ?
	\end{equation}
}

When $q=1$, the chord log-Minkowski problem is just the log-Minkowski problem.
When $q=0$, it is a log-Christoffel-Minkowski problem, which may well be even more difficult
to solve than the still open Christoffel-Minkowski problem.

The partial differential equation associated with \eqref{0.16} is a new type of Monge-Amp\`ere equation on $\sn$,
\begin{equation}\label{0.17}
	\det\big(\nabla_{ij}h + h\delta_{ij})=\frac{f}{h\, \wt V_{q-1}(h, \nabla h)},
	\ \ \ \text{on $\sn$.}
\end{equation}

\vspace{5pt}

In this paper, we will solve the chord Minkowski problem for $q>0$.

\vspace{5pt}
\begin{theo}\label{CMP}
	Suppose real $q>0$ is fixed.
	If $\mu$ is a finite non-zero Borel measure on $\sn$, then there exists a convex body $K\subset\rn$ so that
	\[
	F_q(K,\cdot) = \mu,
	\]
	if and only if, $\mu$ is not concentrated on a
	closed hemisphere of $\sn$, and
	\[
	\int_{\sn} v\, d\mu(v) =0.
	\]
\end{theo}

It will be interesting to explore the extreme case $q=0$ of the chord Minkowski problem by
letting $q\to 0$. This requires understanding the convergence of the solutions
$K_q$ as $q\to 0$ and
the convergence of the measures $F_q(K_q,\cdot)$
as $q\to 0$. Understanding this may shed light on the unsolved
Christoffel-Minkowski problem
for the area measure $S'$.

Similar to the situation with the log-Minkowski problem versus the classical Minkowski problem,
the chord log-Minkowski problem is a harder problem than the chord Minkowski problem.
We shall present sufficient conditions to solve the symmetric case when $1\le q\le n+1$.

\begin{theo}\label{CMP2}
	Suppose $1\le q\le n+1$.
	If $\mu$ is an even, finite, non-zero Borel measure on $S^{n-1}$, then
	there exists an origin-symmetric convex body $K\subset\rn$ such that
	\[
	G_q(K,\cdot) = \mu
	\]
	provided $\mu$ satisfies the subspace mass inequality:
	\begin{equation}
		\frac{\mu(\xi_k\cap S^{n-1})}{|\mu|}  <  \frac{k+\min\{k,q-1\}}{n+q-1},
	\end{equation}
	for for each $k$-dimensional subspace $\xi_k\subset\rn$ and each $k=1, \ldots, n-1$.
\end{theo}
The solution of the chord log-Minkowski problem for the case where $q=1$, when $\mu$ is an even measure,
was obtained in \cite{BLYZ13jams}.

We use variational methods to establish Theorems \ref{CMP} and \ref{CMP2}. One significant difficulty
is in establishing a differential formula for chord integrals that would allow the Minkowski and log-Minkowski
problems for chord and cone-chord measures to be tackled by solving optimization problems. Solving the
optimization problem associated with the chord Minkowski problem turns out to be fairly standard. However, solving
the optimization problem associated with the cone-chord Minkowski problem is not. This is because
cone-chord measures have delicate measure concentration properties, and these create substantial difficulties. Studying the concentration
of cone-chord measures is fundamental for understanding the behavior of such geometric measures.

An important problem that deserves attention, but not considered in this paper, is
the uniqueness problem for chord measures and cone-chord measures: {\it Are the solutions unique?}

\section{ Preliminaries}

Depending on the context, $|\cdot|$ may denote absolute value, or the standard norm on
Euclidean $n$-space, $\rn$, or the length of a line segment, and occasionally the total mass of a Borel measure on $\sn$.

Let $\mathcal C^n$ denote the collection of compact convex subsets of $\rn$. The collection of
{\it convex bodies}, denoted by $\mathcal K^n$, are the members of
$\mathcal C^n$ that have nonempty interiors.
The Hausdorff metric on $\mathcal C^n$ can be defined by
\[
d_\infty(K,L)=\max\{|h_K(v)-h_L(v)| : v\in \sn\}, \ \ \ K, L\in \mathcal C^n.
\]
The fact that the metric space $(\mathcal C^n,d_\infty)$ is locally compact is known as the
{\it Blaschke selection theorem.}
Let $\mathcal K^n_o$ be the collection of all convex bodies that contain the origin as an interior point,
and let $\mathcal K^n_e$ consists of the bodies in $\mathcal K^n_o$ that are symmetric about the origin.

For a subspace $\xi$ in $\rn$, let $P_\xi : \rn \to \xi$ be the  orthogonal projection onto $\xi$, and write
$K|\xi$ or $P_\xi K$ for the image of the convex set $K$ under $P_\xi$.
If $K\in\mathcal C^n$ is a subset of an  $i$-dimensional affine flat, for its $i$-dimensional volume we will write $\vv_i(K)$. If $i=n$ we write $V(K)$ and if $i=n-1$ we write $\vv(K)$.

Let $K\in\mathcal K^n$. For each $v\in \sn$, the hyperplane
\begin{equation*}
	H_K(v) = \{x \in \rn : x\cdot v = h_K(v)\}
\end{equation*}
is called the {\it supporting hyperplane to $K$ with unit normal $v$}.

For $\sigma\subset\partial K$, the {\it spherical image of $\sigma$} is defined by
\begin{equation*}
	\bu_K(\sigma) = \{v\in\sn : x\in H_K(v)\ \text{for some}\ x\in\sigma \} \subset\sn.
\end{equation*}
For $\eta\subset\sn$, the {\it reverse spherical image of $\eta$} is defined by
\[
\bx_K(\eta)=\{ x\in \partial K : x\in H_K(v) \text{ for some } v\in \eta\} \subset \partial K.
\]

Let $\sigma_K\subset \partial K$ be the set consisting of all $x\in \partial K$, for which the set
$\bu_K(\{x\})$, which we frequently abbreviate as $\bu_K(x)$, contains more than a single element.
It is well known that $\hm (\sigma_K) = 0$
(see Schneider \cite{S14}, p.\ 84).
The function
\begin{equation}\label{sim}
	\nu_K : \partial K \setminus \sigma_K \to \sn,
\end{equation}
defined by letting $\nu_K(x)$ be the unique element in $\bu_K(x)$, for each
$x \in \partial K \setminus \sigma_K $, is called the {\it spherical image map}
(also called the Gauss map) of $K$
and is known to be continuous (see Lemma 2.2.12 of Schneider \cite{S14}).
The set $\partial K \setminus \sigma_K$ is often abbreviated by
$\partial'\negthinspace K$.

\vspace{5pt}

\di{\bf Radial and Support Functions.}\
For $K\in\kn$  and $z\in\rn$, let
\[
S_z=S_z(K) =\{u\in \sn: K\cap (z+\mathbb Ru)\neq \varnothing \}.
\]
Note that the set $\{z+u: K\cap (z+\mathbb Ru)\neq \varnothing, \ u\in \rn \}$
is just the set of lines passing through $z$ and meeting $K$, and the intersection
of the translate (by $-z$) of this set with the unit sphere is $S_z$.  The set $S_z$ is
$\sn$ if $z\in K$ and is a the double cap on $\sn$  if $z \notin K$.

The (extended) radial function.
$\rho\lsub{K,z}: \rn \setminus\{0\} \to \mathbb R$,
of $K$ with respect to $z\in\rn$,  is defined for for, $x\in\rn\setminus\{0\}$, by
\begin{equation}\label{c1}
	\rho\lsub{K,z}(x) = \max \{\lambda\in\mathbb R : \lambda x \in K-z\}.
\end{equation}
If $\lambda x\notin K-z$, for all $\lambda\in\rbo$, then $\rho\lsub{K,z}(x)$ is defined to be 0.
Of course, the function $\rho\lsub{K,z}$
may assume negative values, but if $z\in\interior K$ then $\rho\lsub{K,z}$ is clearly positive.
Let
\[
S_z^+=S_z^+(K) = \{u\in \sn : \rho\lsub{K,z}(u) > 0\}, \ \ \  S_z^- =S_z^-(K)=\{u\in \sn : \rho\lsub{K,z}(u)  < 0\}.
\]
If $z\in K$, then obviously $S_z^-=\varnothing$.
Abbreviate $\rho\lsub{K,0}$ by $\rho\lsub{K}$.  For $\hm$-almost all $z\in \partial K$,
(in particular, for those $z\in \partial K$ that have only one outer normal) $\rho\lsub{K,z}$
is positive in an open half-space (orthogonal to the normal of $K$ at $z$) and vanishes in the
interior of the complement  half-space.
If $z\notin K$, the radial function $\rho\lsub{K,z}$ is positive in the interior of one of the two
halves of the double cone and is negative in the interior of the other half.

The radial function $\rho\lsub{K,z}$ is homogeneous of degree $-1$,
\[
\rho\lsub{K,z}(ax) =a^{-1} \rho\lsub{K,z}(x), \ \ \ \ \ a>0.
\]
Trivially the definition of the radial shows that
\begin{equation}\label{c1.0}
	\rho\lsub{K,y+z}= \rho\lsub{K-y,z}= \rho\lsub{K-y-z}, \ \ \ \ \ y, z \in \rn.
\end{equation}

The support function $h_{K,z}: \rn \to \R$ of $K$ with respect to $z$,
is defined by
\begin{equation}\label{c1.1}
	h_{K,z}(x) = \max\{x \cdot y: y\in K - z\}, \ \ \ x \in \rn.
\end{equation}
As a mapping $(z,x) \mapsto h_{K,z}(x)$ the support function is continuous.
For the support function of $K$ with respect to the origin we shall frequently write $h_K$ rather than $h_{K,0}$. Trivially,
\[
h_{K,y+z}(x)=h_{K-y,z}(x)= h_K(x) -(y+z)\cdot x, \ \ \ \ \ x, y, z \in \rn.
\]

From the definition of the radial function, we see that $z+\rho\lsub{K,z}(u)u \in \partial K$,
for each $u\in S_z$. The definition of the support function now shows that for all $v\in \sn$,
\begin{align*}
	v\cdot (z+\rho\lsub{K,z}(u)u)&\le h_K(v),
	\shortintertext{that is,}
	\rho\lsub{K,z}(u)u\cdot v&\le h_{K,z}(v),
\end{align*}
with equality if and only if $v\in \bu_K(y)$,
where $y=z+\rho\lsub{K,z}(u)u$. When equality occurs,
$\rho\lsub{K,z}(u)$ and $h_{K,z}(v)$ have the same sign, and thus $u\cdot v\ge 0$. Moreover, $u\cdot v>0$
if $u$ is not on the boundary of $S_z$.

Obviously,
\[
h_{K,z}(v) = \sup \big\{\rho\lsub{K,z}(u)u\cdot v : u\in S_z \big\},
\]
and
\begin{equation}\label{c1.1.1}
	\rho\lsub{K,z}(u) = \inf \Big\{\frac{h_{K,z}(v)}{u\cdot v} : u\cdot v>0, v\in \sn\Big\}.
\end{equation}
From the definition of radial function we see that that $(z,u)\mapsto\rho\lsub{K,z}(u)$ is continuous on the set
\[
K_{\text{icone}}=\{(z,u)\in \rn\times\sn : (\text{int} K)\cap (z+\R u)\neq \varnothing\}.
\]
In particular, $(z,u)\mapsto\rho\lsub{K,z}(u)$ is continuous on int$K \times \sn$.
At a point $(z,u)$ at which the line $z+\R u$ intersects only the boundary $\partial K$,
the radial function may not be continuous.

\vspace{3pt}

The function $\partial K \times \sn \ni (z,u) \mapsto \rho\lsub{K,z}(u)$
is upper semi-continuous on
$\partial K \times \sn$.
To see this, note that for $c > 0$,
\begin{align*}
	\{(z, u) &\in \partial K \times \sn : \rho\lsub{K,z}(u) \ge c \} \\
	&=\Big\{(z, u) \in \partial K \times \sn :
	|w-z|\ge c \text{ and } \frac{w-z}{|w-z|}=u \text{ for some } w\in \partial K\Big\}.
\end{align*}
This is a closed set because $\partial K$ is compact. When $c=0$, the set on the left-hand side of the equation
is obviously $\partial K \times \sn$.

Since $\partial K \times \sn \ni (z,u) \mapsto\rho\lsub{K,z}(u)$ is upper semi-continuous, it is a Borel function on $\partial K \times \sn$ and the set
\[
D^+ =\{(z,u) \in \partial K \times \sn : \rho\lsub{K,z}(u)>0\},
\]
is a Borel set in $\partial K \times \sn$.

\vspace{5pt}

\di{\bf Dual Quermassintegrals.}\

For $q\in \mathbb R$, define
\begin{equation}\label{2def}
	\wt V_q^+ (K,z)=\frac1n \int_{S_z^+} \rho\lsub{K,z}(u)^q\, du, \ \ \ \ \ \
	\wt V_q^- (K,z)=\frac1n \int_{S_z^-} |\rho\lsub{K,z}(u)|^q\, du.
\end{equation}

When $z\in K$, then $S_z^-=\varnothing$ and $\wt V_q^-(K,z)=0$.

For $q\in \mathbb R$, define
the $q$-th {\it dual quermassintegral}, $\wt V_q (K,z)$, of $K$ with respect to $z\in K$, by,
\begin{equation}\label{c2}
	\wt V_q (K,z) = \wt V_q^+ (K,z).
\end{equation}
For $q>0$, and $z\notin K$ define $\wt V_q (K,z)$ by
\begin{equation}\label{c2.0}
	\wt V_q (K,z) = \wt V_q^+ (K,z) - \wt V_q^- (K,z).
\end{equation}

When $z=0$
the $q$-th dual quermassintegral with respect to the origin is usually written as $\wt V_q(K)$.
Note that this is for all $q\in\R$.

When $q=0$, we have
\begin{equation}\label{c2.2}
	\wt V_0(K,z)= \begin{cases} \omega_n  &z\in K\setminus \partial K \\
		{\omega_n}/2  &\text{for $\hm$-almost all $z\in\partial K$}  \\  0  &z\notin K. \end{cases}
\end{equation}

When $q=n$, the dual quermassintegral $\wt V_n(K,z) =V(K)$ for all $z\in \rn$.

\vspace{3pt}

Since $S_z^+=\sn$ when $z\in \Int K$ and $\rho\lsub{K,z}(u)=0$ when $z\in \partial K$ and $u$ is
in the interior of $\sn \setminus S_z^+$, we have
\begin{equation}\label{c2.1}
	\wt V_q(K,z) = \frac1n \int_{\sn} \rho\lsub{K,z}(u)^q\, du,
\end{equation}
when either $q\in\R$ and $z\in \Int K$, or $q>0$ and $z\in \partial K$.

When $q\in \mathbb R$,
the dual quermassintegral
$\wt V_q(K,\cdot)$ is continuous
on int$K$, but it may be infinite if $z\in \partial K$ and $q$ is negative.
In fact, for $q>0$,
the dual quermassintegral
$\wt V_q(K,\cdot)$ is continuous on $K$.
This follows from \eqref{c2.1}, the bounded convergence theorem, and the
continuity of function $K\ni z \mapsto \rho\lsub{K,z}(u)$ for any fixed $u$.
For $q\le0$, the
dual quermassintegral
$\wt V_q(K,\cdot)$ is lower semi-continuous on $K$.
This follows from \eqref{2def}, Fatou's lemma, the
continuity of function $K\ni z \mapsto\rho\lsub{K,z}(u)$ for any fixed $u$, and
the inclusion $S_z^+ \subseteq \liminf_{k\to\infty} S_{z_k}^+$ for any
sequence $z_k \to z$ in $K$. In particular, for $q\in \R$, the
dual quermassintegral
$\wt V_q(K,\cdot)$ is a Borel function on $\partial K$.

For $q>0$ and fixed $z\in\rn$, it can also be shown that
$\wt V_q(\cdot,z)$ is a continuous function on $\kn$.

When $q>0$, the dual quermassintegral is the Riesz potential of the characteristic function
$\mathbf{1}_K$,
\begin{equation}\label{c2+}
	\wt V_q (K,z) = \frac q{n}\int_{K}|x-z|^{q-n}dx,\ \ \ \ \ q>0,
\end{equation}
which is easily obtained by switching to polar coordinates.

When the index $q$ is an integer between $1$ and $n$, the dual quermassintegrals have the well-known integral representation as means of areas (i.e., lower dimensional volumes) of cross sections through a point,
\begin{equation}\label{c2.3}
	\wt V_i(K,z) =\frac{\omega_n}{\omega_i} \int_{G_{n,i}} \vv_i(K\cap(z+\xi_i))\, d\xi_i, \ \ \ \ \ i=1,2, \ldots, n
\end{equation}
where $\wt V_n(K,z)$ is just $n$-dimensional volume $V(K)$. If in the right-hand side of \eqref{c2.3} intersections with subspaces are replaced by projections onto subspaces, the left-hand becomes the classical quermassintegral.

\di{\bf Spherical Image, Surface Area Measure, and Cone Volume Measure.}\

From the fact that $S(K, \eta) =\mathcal H^{n-1}(\nu_K^{-1}(\eta))$ for each Borel $\eta \subset \sn$, it follows that
for each continuous $g:\sn\to \R$,
\begin{equation}\label{xsm1}
	\int_{\partial'\negthinspace K} g(\nu_K(x))\, d\hm(x) = \int_{\sn} g(v)\,dS(K,v).
\end{equation}

The cone volume measure $V(K, \cdot)$ of a convex body $K\in \kno$ is easily defined by using the spherical image map
$\nu_K:\partial'\negthinspace K \to\sn$.
For each Borel set $\eta \subset \sn$,
\begin{equation}\label{cm}
	V(K, \eta) = \frac1n \int_{\nu_K^{-1}(\eta)} (x\cdot \nu_K(x))\, d\mathcal H^{n-1}(x).
\end{equation}
From this definition and \eqref{xsm1} it is easily shown that,
\begin{equation}\label{cm1}
	dV(K, \cdot) = \frac1n h_K \, dS(K,\cdot).
\end{equation}
Cone volume has received increasing attention in the recent past. See e.g., \cite{BH16adv, BH17adv,  BLYZ13jams, CLZ, N07tams, NR03aipps, PW12plms, Sta1, Sta2, X1}.

\vspace{5pt}
\di{\bf The Brunn-Minkowski Inequality.}\
For $K, L\in \cn$ and $t\ge 0$, the Minkowski combination $K+tL\in \mathcal C^n$ can be defined by
\[
K+tL=\{x+ty : x\in K,\ y\in L\}.
\]
The Minkowski combination $K+tL$ has a support function given by
$h_{K+tL} = h_K + th_L$.

A real-valued function $\Phi$ on $\mathcal C^n$, or one of its subsets, is said to be homogeneous of degree $\lambda\in\rbo$ if $\Phi(tK)=t^\lambda\Phi(K)$ for all $K$ in its domain and all $t> 0$.

The volume functional, $V:\in \mathcal C^n \to [0,\infty)$, is not only homogeneous of degree $n$ but it satisfies the {\it Brunn-Minkowski inequality}: For $K,L\in \mathcal C^n$ and $t\ge 0$,
\[
V(K+tL)^{1/n} \ge V(K)^{1/n} + tV(L)^{1/n}.
\]

\section{ Chord Integrals}

Throughout this section we assume $K\in\kn$ is an arbitrary but fixed convex body. For quick reference we repeat, and slightly extend, the definition of the chord-integral,
\begin{equation}\label{c4}
	I_q(K) = \int_{K\cap \ell\neq\varnothing}
	|K\cap \ell|^q\, d\ell, \ \ \ \ \ \text{$q>-1$},
\end{equation}
where the integration is
over all $\ell\in\Ln=G^a_{n,1}$ that meet $K$ and is
with respect to the invariant Haar measure on $\Ln$.
It is easily seen (by imitating  \eqref{c6.4} and \eqref{c6.5} below) that, with respect to the invariant measure on $\Ln$, for almost all $\ell$
such that $K\cap \ell\neq\varnothing$ it is the case that $|K\cap \ell|>0$.

The measure in \eqref{c4} is normalized so that
it is a probability measure when restricted to rotations. More specifically, since a line $\ell\in\Ln$ is uniquely
associated with an antipodal pair $u,-u\in\sn$ and the intersection point $x=u^\perp \cap \ell$, where
$u^\perp$ is the $(n-1)$-dimensional subspace orthogonal to $u$ (i.e.,
$\ell = x+ \mathbb R u$)
the measure $d\ell$ is normalized so that
\begin{equation}\label{c3}
	d\ell = \frac1{n\omega_n} dx du,
\end{equation}
where
$d\ell$ is the measure on $\Ln$, while
$dx$ is Lebesgue measure in $u^\perp$ and $du$ is spherical Lebesgue measure
on $\sn$. Note that the normalization of our measure is such that for the unit ball $I_1(B^n)=V(B^n)=\omega_n$.

Thus the chord integral $I_q(K)$ is given by
\begin{equation}\label{c6.1}
	I_q(K) = \frac1{n\omega_n} \int_{\sn} \int_{K|u^\perp}
	X_{K}(x,u)^q\, dxdu, \ \ \ \ \ q>-1,
\end{equation}
where $K| u^\perp$ is the image of the orthogonal projection of $K$ onto $u^\perp$, and where
the {\it parallel $X$-ray} of $K$ is the non-negative function on
$\rn\times\sn$ defined by
\begin{equation}\label{c6}
	X_{K}(z,u) = |K\cap (z+\mathbb Ru)|, \ \ \ \ \ \ z \in \rn, \ u\in \sn.
\end{equation}

The $X$-ray $X_{K}(x,u)$ and the extended radial function $\rho\lsub{K,z}(u)$ have the following relation,
\begin{equation}\label{c6.2}
	X_{K}(x,u) = \rho\lsub{K,z}(u) + \rho\lsub{K,z}(-u), \ \  \text{when} \
	K\cap (x+\mathbb R u) = K\cap (z+\mathbb R u) \neq \varnothing.
\end{equation}
When $z\in \partial K$, then either $\rho\lsub{K,z}(u)=0$ or $\rho\lsub{K,z}(-u)=0$ for almost all
$u\in \sn$, and thus
\begin{equation}\label{c6.6}
	X_{K}(z,u) = \rho\lsub{K,z}(u), \ \ \text {or }\ \ X_{K}(z,u) = \rho\lsub{K,z}(-u), \ \ \ \ z\in \partial K,
\end{equation}
for almost all  $u\in \sn$.

It will be convenient to define $I_q(K, u)$, for $u\in\sn$ as
\begin{equation}\label{c6.4}
	I_q(K, u) =  \int_{K|u^\perp} X_{K}(x,u)^q\, dx,
\end{equation}
so that
\begin{equation}\label{c6.5}
	I_q(K) =\frac1{n\omega_n} \int_{\sn} I_q(K,u)\, du.
\end{equation}

\begin{lemm}\label{I11}
	For each $q>-1$ and each $K\in\kn$, the chord-power integral
	$I_q(K)$ is finite.
\end{lemm}

\begin{proof} Obviously, $I_q(K)$ is finite when $q\ge 0$.
	For the case of $-1<q<0$, without loss of generality we assume
	that the origin is in the interior of $K$ and let
	$b={\rho\llsub{\text{i}}}/{\rho\llsub{\text{o}}}$, where $\rho\lsub{\text{i}}$ and $\rho\lsub{\text{o}}$
	are the minimum and the maximum of $\rhok$. For fixed   $u\in\sn$ consider the Steiner
	symmetrization $S_uK$ of $K$ with respect to $u^\perp$ and observe that
	$X_{S_uK}(x,u)= X_{K}(x,u)$, for all
	$x\in K|u^\perp$. Note that since $\rho\lsub{\text{i}}B \subset K \subset \rho\lsub{\text{o}}B$,
	where $B=B^n$ the unit ball centered at the origin,
	$\rho\lsub{\text{i}}B \subset S_uK \subset \rho\lsub{\text{o}}B$.
	Let $C_uK$ denote the double cone
	inside $S_uK$ spanned by $K|u^\perp$ and $S_uK\cap (\mathbb R u)$. Then for all $x\in K|u^\perp$,
	\begin{multline*}
	X_{K}(x,u)
	= X_{S_uK}(x,u)
	\ge X_{C_uK}(x,u)\\
	=
	(\rho\lsub{S_uK}(\langle x \rangle) - |x|)
	\frac{X_{S_uK}(x,u)}{\rho\lsub{S_uK}(\langle x \rangle)}
	\ge
	(\rho\lsub{S_uK}(\langle x \rangle) - |x|)\,
	\frac{2\rho\lsub{\text{i}}}{\rho\lsub{\text{o}}},
	\end{multline*}
	where $\langle x \rangle$ is used in place of $x/|x|$.

	By using \eqref{c6.1}, and switching to polar coordinates for the inner integral by
	writing $x\in K|u^\perp$ as $x=rv$ with $0\le r\le \rho\lsub{S_uK}(v)$, and $v\in S^{n-1}\cap u^\perp$,
	we have
	\begin{align*}
		I_q(K) &= \frac1{n\omega_n} \int_{\sn} \int_{K|u^\perp} X_{K}(x,u)^q\, dxdu \\
		&\le \frac{2^qb^q}{n\omega_n} \int_{\sn} \int_{S^{n-1}\cap u^\perp}\int_0^{\rho\llsub{S_uK}\hspace{-1pt}(v)}
		(\rho\lsub{S_uK}(v)-r)^q r^{n-2}drdvdu \\
		&= \frac{2^qb^q}{n\omega_n}\int_{\sn} \int_{S^{n-1}\cap u^\perp} \rho\lsub{S_uK}(v)^{n+q-1} dvdu
		\int_0^1 (1-t)^q t^{n-2}dt\\
		&<\ \infty.
	\end{align*}
\end{proof}

An important property of the functional $I_q$ is its homogeneity.

\begin{lemm}\label{Lud03, }
	If  $K\in \kn$ and $q>-1$, then
	\begin{equation}\label{homog}
		I_q(tK) = t^{n+q-1}I_q(K),
	\end{equation}
	for $t>0$.
\end{lemm}

Formulas in the following lemma are known but with different normalization.
The first formula was proved in \cite{Z99tams} (Lemma 3.1), see also, see \cite{GZ98} (Lemma 2.1).
The second formula is classical, see \cite{Ren, San}.
For completeness, we include a proof here.

\begin{lemm}\label{I1}
	If $K\in \kn$, then
	\begin{align}\label{c6.3}
		I_q(K) &= \frac{q}{\omega_n} \int_K \wt V_{q-1}(K,z) \, dz, \ \  q>0, \\
		I_q(K) &= \frac{q(q-1)}{n\omega_n}
		\int_{x\in K} \int_{z\in K}
		|x-z|^{q-n-1}dxdz, \ \ q>1.  \label{c6.3.1}
	\end{align}
\end{lemm}

\begin{proof}
	Suppose $u\in \sn$ and $x\in K | u^\perp$. For $z\in K\cap(x+\rbo u)$, define $s(z)=s\in\rbo$ by $z=x+su$ and let $s_0, s_1$ be the minimum and maximum of $s$. The inverse map $s^{-1}: [s_0,s_1] \to K\cap (x+\rbo u)$ (where $s\mapsto x+su$) is obviously a bijection.
	Observe that $X_K(x,u) = s_1-s_0$ while $s_1 - s(z)= \rho\lsub{K,z}(u)$.
	Then by \eqref{c6.1} and \eqref{c2},
	we have
	\begin{align*}
		I_q(K) &= \frac1{n\omega_n} \int_{\sn} \int_{u^\perp} X_K(x,u)^q\, dxdu \\
		&=\frac q{n\omega_n}\int_{\sn}\int_{K|u^\perp} \int_{s_0}^{s_1} (s-s_0)^{q-1}\, ds dxdu \\
		&=\frac q{n\omega_n}\int_{\sn}\int_{K|u^\perp} \int_{s_0}^{s_1} (s_1-s)^{q-1}\, ds dxdu \\
		&=\frac q{n\omega_n}\int_{\sn}\int_K \rho\lsub{K,z}(u)^{q-1} \, dz du \\
		&=\frac {q}{\omega_n} \int_K \wt V_{q-1}(K,z) \, dz.
	\end{align*}
	This gives \eqref{c6.3} which together with \eqref{c2+} give \eqref{c6.3.1}.
\end{proof}

Although we focus almost exclusively on the chord-integral
$I_q:\kn\to [0,\infty)$, for $q\ge 0$ there is an obvious extension $I_q:\cn\to [0,\infty)$
\begin{equation}\label{c99}
	I_q(K) = \int_{K\cap \ell\neq\varnothing}
	|K\cap \ell|^q\, d\ell,\qquad \text{$q\ge 0$\  and\  $K\in\cn$},
\end{equation}
where the integration is with respect to the invariant Haar measure on $\Ln$.
The continuity of the functional $I_q:\cn\to [0,\infty)$ will be exploited.

\section{  Chord Measures}

The purpose of this section is to define chord measures and cone-chord measures, and to establish some important basic results about these two measures. Also established in this section is the surprising connection between mean curvature and dual mixed volumes.

\begin{defi}
	Let $K\in \kn$ and $q>0$. The {\it chord measure} $F_q(K,\cdot)$ is
	a finite Borel measure on $\sn$ defined by
	\begin{equation}\label{c11}
		F_q(K,\eta) = \frac{2q}{\omega_n} \int_{\nu_K^{-1}(\eta)} \wt V_{q-1}(K, z)\, d\hm(z),
		\ \ \ \ \text{Borel}\  \eta\subset \sn.
	\end{equation}
	For $K\in\kn$ containing the origin (but not necessarily in its interior) the {\it cone-chord measure} $G_q(K,\cdot)$ is
	a finite Borel measure on $\sn$ defined by
	\begin{equation}\label{c11.1}
		G_q(K,\eta) = \frac{2q}{(n+q-1)\omega_n} \int_{\nu_K^{-1}(\eta)} (z\cdot \nu_K(z))
		\wt V_{q-1}(K, z)\, d\hm(z), \ \ \ \ \text{Borel}\ \eta\subset \sn.
	\end{equation}
\end{defi}

\begin{rema}
	\leavevmode\newline
	\rm\phantom{X}
	\dr{a} When $q=1$ we know $\wt V_0(K, z)={\omega_n}/2$ for $\hm$-almost all $z\in \partial K$. Thus,
	\begin{equation}\label{c12}
		F_1(K,\cdot) = S(K, \cdot),
		\ \ \  \text{and} \ \ \ \ G_1(K,\cdot) = V(K, \cdot).
	\end{equation}
	Thus, the family of chord measures includes the surface area measure as a special case,
	and the family of cone-chord measures includes the cone-volume measure as a special case.

	\dr{b} We know $\wt V_n(K,z) =V(K)$, for all $z\in K$, thus when $q=n+1$, we have
	\begin{equation}\label{c13}
		F_{n+1}(K,\cdot) = \frac{2(n+1)}{\omega_n} V(K) S(K, \cdot), \ \ \  \text{and} \ \ \ \
		G_{n+1}(K,\cdot) = \frac{n+1}{\omega_n} V(K) V(K,\cdot).
	\end{equation}
	Since the functionals $K\mapsto S'(K, \cdot)$ and  $K\mapsto S(K,\cdot)$ are both additive,\negthinspace\footnote{Recall that a functional $f$ on convex sets is {\it additive} provided $f(K\cup L) + f(K\cap L) = f(K) + f(L)$ for all convex sets $K$ and $L$, whenever $K\cup L$ is convex.}
	special cases of the chord measures, $F_0(K,\cdot)$, $F_1(K, \cdot)$, and $\frac1{V(K)}F_{n+1}(K,\cdot)$
	are also additive. It would be interesting to characterize these additive functionals among the whole family of chord measures. The valuation theory that studies additive functionals over convex bodies (or related sets) has received much attention in recent years, see e.g., \cite{Ales1, Ales2, BL16, H12, HP14crelle, Lud04, Lud10, LR10annals, SW16jems}

	\dr{c} If $K$ is a polytope, it is easily seen that the chord measure $F_q(K,\cdot)$ and the cone
	chord-measure $G_q(K,\cdot)$ are discrete measures that concentrate on the normal directions
	of the facets of the polytope $K$.
	
	\dr{d} From \eqref{c1.0} we see that $\wt V_q(K, z) = \wt V_q(K-y, z-y)$, for $y\in \rn$, which implies that
	$F_q(K-y, \cdot) = F_q(K, \cdot)$; i.e., the chord measure $F_q(K,\cdot)$ is invariant under translations of $K$.
	
	\dr{e} From the definition of chord and cone-chord measures it follows that for each bounded Borel function $g: \sn \to \R$, we have
	\begin{align}
		\int_{\sn} g(v)\, dF_q(K,v)
		&=\frac {2q}{\omega_n} \int_{\partial K}  \wt V_{q-1}(K,z) g(\nu_K(z))\, d\hm(z),\label{XX1} \\
		\shortintertext{and}
		\int_{\sn} g(v)\, dG_q(K,v)
		&=\frac{2q}{(n+q-1)\omega_n} \int_{\partial K} (z\cdot \nu_K(z))
		\wt V_{q-1}(K, z) g(\nu_K(z))\, d\hm(z).\label{XX2}
	\end{align}
	\noindent
	The formulas can be established using
	a standard argument of approximation by simple functions.
	
	\dr{f} The formula \eqref{0.2} is extended to the following,
	\begin{equation}\label{cm2}
		dG_q(K, \cdot) = \frac1{n+q-1} h_K \, dF_q(K,\cdot).
	\end{equation}
	To see this simply replace $g:\sn\to\rbo$ by $gh_K:\sn\to\rbo$ in \eqref{XX1} and compare with \eqref{XX2}.
	
	\dr{g} The move from the Bruun-Minkowski theory to the $L_p$-Brunn-Minkowski theory (see  \cite{L93jdg, L96adv}) has proven to be fruitful. Within the $L_p$-theory the classical surface area measure $S$ becomes the $L_p$-surface area measure $S_p$,
	while the area measure $S'$ became the $L_p$-area measure $S'_p$. With $L_p$-surface area measure one could prove stronger $L_p$ affine isoperimetric inequalities \cite{HS09jdg, LYZ00jdg} and, when combined with the solution to the even $L_p$-Minkowski problem, these lead to stronger $L_p$-affine Sobolev inequalities \cite{HS09jfa, LYZ02jdg, Z99jdg}.
	Having $L_p$-surface area measure in hand, it was not difficult to see how classical affine surface area (from Blaschke's School), $\Omega$, can be extended to $L_p$-affine surface area $\Omega_p$. The $L_p$-affine surface area connected classical affine surface area ($p=1$) with classical centro-affine surface area ($p=n$). See e.g., \cite{L96adv, MW00adv, SW04adv, Wadv12, WYma10, Zhao15imrn} to where some of this lead.
	
	We now do much the same for chord measures $F_q$.
	For each $K\in\kno$ and
	for each $p\in \R$ define the {\it $L_p$ chord measures} by
	\begin{equation}\label{c11.2}
		F_{p,q}(K,\eta) = \frac{2q}{\omega_n} \int_{\nu_K^{-1}(\eta)}
		(z\cdot \nu_K(z))^{1-p}\wt V_{q-1}(K, z)\, d\hm(z),
		\ \ \ \
		\text{Borel}\ \eta\subset \sn.
	\end{equation}
	Now both chord measures and cone-chord measures arise as special cases. For the $L_p$-surface area measure the most important special cases turned out to be $p=-n,0,1,2,n$; see e.g., \cite{CW06adv, BLYZ13jams, L96adv, LYZ00duke, LYZ02duke, Lud03, Hug96}. Might different values of $p$ play prominent roles for $L_p$ chord measures?

	\dr{h} Of course,
	\[
	F_{p,q}(K,\cdot) = h_K^{1-p} F_q(K,\cdot).
	\]
	With the concept of $L_p$ chord measures, the chord Minkowski and log-Minkowski problems
	are unified within the {\bf $L_p$ chord Minkowski problem}:
	
	{\it
		If $\mu$ is a finite Borel measure on $\sn$, and  $p,q \in\R$ with $q\ge0$,
		what are necessary and sufficient conditions so that there exists a convex body
		$K\in \kno$ that solves the equation,
		
		\vspace{-.3cm}
		
		\begin{equation}
			F_{p,q}(K,\cdot) = \mu\, ?
		\end{equation}
	}

	Note that for $q=1$ the $L_p$ chord Minkowski problem becomes the $L_p$ Minkowski problem, a problem
	which has been extensively studied during the last two decades, see e.g.,
	\cite{BLYZ13jams, CLZ, CW06adv, HLYZ05dcg, LW13, L93jdg, Sta1, Zu, Zu2}.
	
	While it is obviously of interest to study the $L_p$ chord Minkowski problem, in this paper,
	we shall exclusively focus on the two critical cases $p=1$ and $p=0$.
\end{rema}

Since $\wt V_{q-1}(K,\cdot)$ is nonnegative and semi-continuous on $\partial K\setminus \sigma_K$,
it is $\hm$-measurable on $\partial K$.
Thus $F_q(K, \cdot)$ and $G_q(K, \cdot)$ are Borel measures.
From \eqref{2def} and \eqref{c2.0}, we know that $\wt V_{q-1}(K, \cdot)$ is a bounded function when
$q\ge 1$, and thus $F_q(K, \cdot)$ and $G_q(K, \cdot)$ are
finite Borel measures. However, when $0<q<1$, the function $\wt V_{q-1}(K, \cdot)$ may not be bounded
and it is not clear that it is integrable on $\partial K$. Thus, a proof is required that $F_q(K, \cdot)$ is finite when $0<q<1$.

In this section, we will prove that $F_q(K, \cdot)$ is finite for all $q>0$, and moreover, we will establish a formula
for chord integrals in terms of $F_q(K, \cdot)$, which
implies that the measure $F_q(K,\cdot)$ has its centroid at the origin, for all $K\in\kn$. This centroid property
turns out to be precisely the necessary and sufficient condition for the solvability of the chord Minkowski problem.

For the limiting case $q=0$,
when the convex body $K$ has smooth boundary, we show that $q\wt V_{q-1}(K, z)$
converges, as $q\to 0$, to the mean curvature of $\partial K$ at $z$ (times a constant factor).
This implies that,  when $K$ has positive curvature,
the limiting case $q=0$ of the chord measure $F_q(K,\cdot)$  is the area measure
$S'(K,\cdot)$ times a constant factor.

\vspace{5pt}

\begin{theo}\label{I9}
	If $K\in \kn$ and $q>0$, then the Borel measure $F_q(K,\cdot)$ is finite and
	\begin{equation}\label{c14}
		I_q(K) = \frac1{n+q-1} \int_{\sn} h_K(v)\, dF_q(K,v).
	\end{equation}
	If  $K\in\kn$ contains the origin and $q>0$, then
	\begin{equation}\label{c14+}
		I_q(K) =\int_{\sn}  dG_q(K,v).
	\end{equation}
\end{theo}

To prove Theorem \ref{I9}, we require several lemmas.

\begin{lemm}\label{I2}
	If $K\in\kno$, then for $\hnn$-almost all $z\in K$ and for $\hm$-almost all $u\in\sn$,
	\begin{equation}\label{c7}
		z\cdot \nabla \rho\lsub{K,z}(u) = \rho\lsub{K,z}(u) - \frac {y\cdot \nu_K(y)}{u\cdot \nu_K(y)},
	\end{equation}
	where 
    $\nabla \rho\lsub{K,z}(u)$ is the gradient in $\rn$ of the function $z\mapsto \rho\lsub{K,z}(u)$ on $K$, 
    and where $y=z+\rho\lsub{K,z}(u)u$.
\end{lemm}

\begin{proof} Recall $\rhok = \rho\lsub{K,0}$.
	The function $z\mapsto \rho\lsub{K,z}(u)$ on $K$, is easily seen to be concave. Hence it is differentiable $\hnn$-almost everywhere on $K$.
	
	Since $y\in\partial K$, the definition of radial function gives
	\begin{equation}\label{c8}
		\rhok(z+\rho\lsub{K,z}(u)u)=\rhok(y) =1.
	\end{equation}
	Since
	$\rhok$ is homogeneous of degree $-1$, Euler's theorem on homogeneous functions gives,
	\begin{equation}\label{c8.1.1}
		y\cdot \nabla \rhok(y) = -\rhok(y) =-1.
	\end{equation}
	By definition of the radial function, $\rhok(x)=1 \iff x\in\partial K$. From this we conclude that $\nabla \rhok(y)= -a_K(y)\nu_K(y)$, where $a_K(y)>0$; i.e., $\nabla \rhok(y)$ is an inner normal of $\partial K$ at $y$. Substituting $-a_K(y)\nu_K(y)$ for $\nabla \rhok(y)$ in \eqref{c8.1.1} gives
	\begin{equation}\label{c9}
		\nabla \rhok(y) = - \frac1{y\cdot \nu_K(y)} \nu_K(y).
	\end{equation}
	Differentiating the functions in \eqref{c8} with respect to $z$ gives
	\[
	\nabla\rhok(y) (I + u^\intercal\, \nabla\rho\lsub{K,z}(u))=0,
	\]
	where $I$ is the identity matrix. Thus,
	\[
	\nabla \rhok(y) + \big(u\cdot \nabla\rhok(y)\big) \nabla \rho\lsub{K,z}(u)=0.
	\]
	By multiplying both sides by $z^\intercal$ and using \eqref{c9}, we have
	\begin{align*}
		z\cdot \nabla \rho\lsub{K,z}(u) &= -\frac{z\cdot \nabla \rhok(y)}{u\cdot \nabla \rhok(y)} \\
		&=-\frac{(y-\rho\lsub{K,z}(u)u)\cdot \nabla \rhok(y)}{u\cdot \nabla \rhok(y)} \\
		&=-\frac{y\cdot\nabla\rhok(y)}{u\cdot\nabla\rhok(y)} + \rho\lsub{K,z}(u) \\
		&=-\frac{y\cdot\nu_K(y)}{u\cdot\nu_K(y)} + \rho\lsub{K,z}(u).
	\end{align*}
\end{proof}

\begin{lemm}\label{I3}
	If $K\in\kno$, and real $q,\varepsilon>0$, then
	\begin{multline}
		(q-1)\int_{\sn}\int_K (\rho\lsub{K,z}(u)+\varepsilon)^{q-2} \frac{(z+\rho\lsub{K,z}(u)u)\cdot\nu_K(z+\rho\lsub{K,z}(u)u)}{u\cdot \nu_K(z+\rho\lsub{K,z}(u)u)} dz du \\
		= \int_{\partial K} \int_{\sn} y\cdot\nu_K(y) (\rho\lsub{K,y}(v)+\varepsilon)^{q-1} dv d\hm(y)
		- \varepsilon^{q-1} n^2\omega_n V(K).
	\end{multline}
\end{lemm}

\begin{proof} For fixed $z$, the formula for infinitesimal cone volume (see for example,
	\cite{LYZ02duke}, Lemma 2) is
	\begin{equation}\label{gto}
		\rho\lsub{K,z}(u)^n\, du = \rho\lsub{K,z}(u)u\cdot \nu_K(y)\, d\hm_{\partial K}(y),
	\end{equation}
	where $\hm_{\partial K}$ denotes $(n-1)$-dimensional Hausdorff measure restricted to $\partial K$, and  $y=z+\rho\lsub{K,z}(u)u$.
	Now \eqref{gto}, Fubini's theorem, and finally switching to polar coordinates with
	$z-y= rv$, where $r\ge 0$ and $v\in\sn$, give
	\begin{align*}
		(q-1)&\int_K \int_{\sn}(\rho\lsub{K,z}(u)+\varepsilon)^{q-2}
		\frac{y\cdot\nu_K(y)}{u\cdot \nu_K(y)}\, du dz \\
		&=(q-1) \int_K \int_{\partial K} (|z-y|+\varepsilon)^{q-2}
		(y\cdot \nu_K(y)) |z-y|^{1-n}\, d\hm(y) dz \\
		&=(q-1)\int_{\partial K} \int_{\sn} \int_0^{\rho\lsub{K,y}(v)}
		(y\cdot \nu_K(y)) (r+\varepsilon)^{q-2} dr dv d\hm(y) \\
		&=\int_{\partial K} \int_{\sn} (y\cdot \nu_K(y))((\rho\lsub{K,y}(v)+\varepsilon)^{q-1}
		-\varepsilon^{q-1}) \, dvd\hm(y)\\
		&=\int_{\partial K} \int_{\sn} (y\cdot \nu_K(y))(\rho\lsub{K,y}(v)+\varepsilon)^{q-1} \, dvd\hm(y)
		- \varepsilon^{q-1} n^2\omega_n V(K).
	\end{align*}
\end{proof}

\begin{lemm}\label{I5}
	If $K\in\kno$,
	then for $q>0$ and $\varepsilon >0$,
	\[
	\varepsilon \int_{\sn}\int_K (\rho\lsub{K,z}(u) +\varepsilon)^{q-2} dz du \to 0 \ \text{ as }\
	\varepsilon \to 0^+.
	\]
\end{lemm}

\begin{proof} The case of $q\ge 2$ is obvious. We now establish the case where $0<q<1$.
	The case where $1\le q<2$ can be handled in a similar manner.

	We first show that there exists a constant $b>0$ so that
	\begin{equation}\label{I4}
		\rho\lsub{K,z}(u) \ge b(\rhok(v) -r),
	\end{equation}
	where $z=rv\in K$ and $u, v\in \sn$. To see this consider the triangle formed by the origin
	and the two boundary points of $K$, say
	$z_0=\rhok(u)u$ and $z_1=\rhok(v)v$.
	Suppose $z+\rho u$ is on the triangle's side $z_0z_1$,
	then
	\[
	\frac{\rhok(v)-r}{\rhok(v)} = \frac{\rho}{\rhok(u)}.
	\]
	That $\rho\lsub{K,z}(u)\ge \rho$ follows from the definition of the radial function and the fact that $z+\rho u\in K$.
	Thus,
	\[
	\rho\lsub{K,z}(u) \ge \rho = \frac{\rhok(u)}{\rhok(v)} (\rhok(v)-r).
	\]
	Obviously, $b= \min_{\sn} \rho\lsub{K}/\max_{\sn} \rho\lsub{K}>0$ provides the desired lower bound for
	$\rho\llsub{K}(u)/{\rho\llsub{K}(v)}$.

	Let $\rho\lsub{1}$ be the maximum of $\rhok(v)$.
	Using \eqref{I4} and polar coordinates with $z=rv$, where $r \ge 0$ and $v\in \sn$, we have
	\begin{align*}
		\varepsilon \int_K (\rho\lsub{K,z}(u) +\varepsilon)^{q-2} dz
		&\le\varepsilon \int_{\sn} \int_0^{\rho\llsub{K}(v)}
		(b(\rhok(v)-r) +\varepsilon)^{q-2} r^{n-1}drdv \\
		&\le\varepsilon \rho\lsub{1}^{n-1} \int_{\sn}
		\int_0^{\rho\llsub{K}(v)} (b(\rhok(v)-r) +\varepsilon)^{q-2}drdv \\
		&=\frac{\varepsilon \rho\lsub{1}^{n-1}}{b(1-q)} \int_{\sn}
		\big(\varepsilon^{q-1} - (b\rhok(v)+\varepsilon)^{q-1}\big)dv \\
		&\le \frac{\rho\lsub{1}^{n-1}n\omega_n}{b(1-q)}
		\big(\varepsilon^q - \varepsilon (b\rho\lsub{1}+\varepsilon)^{q-1}\big),
	\end{align*}
	and hence,
	\[
	\varepsilon \int_K (\rho\lsub{K,z}(u) +\varepsilon)^{q-2} dz
	\longrightarrow 0,\quad \text{ uniformly as } \varepsilon \to 0^+.
	\]
\end{proof}

\begin{lemm}\label{I6}
	If $K\in\kno$,
	then for $q>0$  and $\varepsilon >0$,
	\begin{multline*}
		\int_{\partial K} \int_{S_y^+} (y\cdot \nu_K(y)) (\rho\lsub{K,y}(v)+\varepsilon)^{q-1}\, d\hm(y)dv \\
		=\frac{n+q-1}2 \int_{\sn}\int_K (\rho\lsub{K,z}(u)+\varepsilon)^{q-1}\, dz du
		+\frac{(1-q)\varepsilon}2 \int_{\sn}\int_K (\rho\lsub{K,z}(u)+\varepsilon)^{q-2}\, dz du.
	\end{multline*}
\end{lemm}

\begin{proof}
	Since the boundary 
	$\partial K$ is line-free (see Schneider \cite{S14}) in direction $u$
	for almost all $u\in S^{n-1}$, we choose such a direction $u$.
	For such a direction $u$, the mapping $z\mapsto z(\rho\lsub{K,z}(u)+\varepsilon)^{q-1}$ is a bounded vector field in $K$,
	and $\nabla\rho\lsub{K,z}(u)$ is bounded except for an $\thh^{n-1}$-null set. Thus
	the divergence theorem (see e.g., \cite{B-M94}) and \eqref{c7}, yield
	\begin{align*}
		\int_{\partial K} &(z\cdot \nu_K(z)) (\rho\lsub{K,z}(u)+\varepsilon)^{q-1} d\hm(z)
		=\int_K \text{div}(z(\rho\lsub{K,z}(u)+\varepsilon)^{q-1}) dz \\
		&=\int_K \big(n (\rho\lsub{K,z}(u)+\varepsilon)^{q-1} +(q-1)(\rho\lsub{K,z}(u)+\varepsilon)^{q-2}
		z\cdot \nabla \rho\lsub{K,z}(u)\big)\, dz \\
		&=\int_K \Big((n+q-1) (\rho\lsub{K,z}(u)+\varepsilon)^{q-1}
		-(q-1)\varepsilon(\rho\lsub{K,z}(u)+\varepsilon)^{q-2} \\
		&\hspace{150pt} -(q-1)(\rho\lsub{K,z}(u)+\varepsilon)^{q-2}
		\frac{y\cdot \nu_K(y)}{u\cdot\nu_K(y)}\Big)\, dz.
	\end{align*}
	We integrate over $u\in\sn$, apply Lemma \ref{I3}, and obtain
	\begin{align*}
		\int_{\sn}\int_{\partial K} &(z\cdot \nu_K(z)) (\rho\lsub{K,z}(u)+\varepsilon)^{q-1} d\hm(z) du \\
		&=\int_{\sn}\int_K \big((n+q-1) (\rho\lsub{K,z}(u)+\varepsilon)^{q-1}
		-(q-1)\varepsilon(\rho\lsub{K,z}(u)+\varepsilon)^{q-2}\big) dzdu \\
		&\hspace{27pt}  - \int_{\partial K} \int_{\sn} y\cdot\nu_K(y)
		(\rho\lsub{K,y}(v)+\varepsilon)^{q-1} dv d\hm(y)
		+ \varepsilon^{q-1} n^2\omega_n V(K).
	\end{align*}
	Therefore,
	\begin{multline*}
		2 \int_{\sn}\int_{\partial K} (z\cdot \nu_K(z)) (\rho\lsub{K,z}(u)+\varepsilon)^{q-1} d\hm(z) du  \\
		=\int_{\sn}\int_K \big((n+q-1) (\rho\lsub{K,z}(u)+\varepsilon)^{q-1}\\
		-(q-1)\varepsilon(\rho\lsub{K,z}(u)+\varepsilon)^{q-2}\big)\, dz du
		+ \varepsilon^{q-1} n^2\omega_n V(K).
	\end{multline*}
	Since $S_z^+$ is a   hemisphere (not necessarily open or closed) for almost all $z\in\partial K$, it follows that
	$\rho\lsub{K,z}(u)=0$ when $u\in S_z\setminus S_z^+$. Thus,
	\begin{multline*}
		\int_{\sn}\int_{\partial K} (z\cdot \nu_K(z)) (\rho\lsub{K,z}(u)+\varepsilon)^{q-1} d\hm(z) du  \\
		=\int_{\partial K}\int_{S_z^+} (z\cdot \nu_K(z)) (\rho\lsub{K,z}(u)+\varepsilon)^{q-1} du d\hm(z)
		+ \varepsilon^{q-1} \frac{n^2\omega_n}2 V(K).
	\end{multline*}
	The desired equation follows.
\end{proof}

\begin{lemm}\label{I7}
	If $K\in\kno$,
	then for all $q>0$,
	\begin{align}
		I_q(K) &= \frac {2q}{(n+q-1)n\omega_n} \int_{\partial K} \int_{S_z^+}
		(z\cdot \nu_K(z)) \rho\lsub{K,z}(u)^{q-1} dud\hm(z) \\
		&=\frac {2q}{(n+q-1)\omega_n} \int_{\partial K} (z\cdot \nu_K(z))
		\wt V_{q-1}(K,z)\, d\hm(z). \label{c10}
	\end{align}
\end{lemm}

\begin{proof} Since $\rho\lsub{K,z}(u)^{q-1}$ is bounded when $q\ge 1$,   we have
	\begin{align*}
		\lim_{\varepsilon\to 0^+} &\int_{\partial K}\int_{S_z^+} (z\cdot \nu_K(z))
		(\rho\lsub{K,z}(u)+\varepsilon)^{q-1} d\hm(z) du\\
		& \hspace{55pt}=\int_{\partial K}\int_{S_z^+} (z\cdot \nu_K(z)) \rho\lsub{K,z}(u)^{q-1} d\hm(z) du, \\
		\lim_{\varepsilon\to 0^+} &\int_{\sn}\int_K (\rho\lsub{K,z}(u)+\varepsilon)^{q-1} dzdu
		=\int_{\sn}\int_K \rho\lsub{K,z}(u)^{q-1} dzdu.
	\end{align*}
	When $0<q<1$, observe that $(\rho\lsub{K,z}(u)+\varepsilon)^{q-1}$ is increasing as  $\varepsilon \downarrow  0$. Thus
	by the monotone convergence theorem, the two limits above still hold.

	Then Lemmas \ref{I5}, \ref{I6} and \ref{I1} give that
	\begin{align*}
		\int_{\partial K}\int_{S_z^+} (z\cdot \nu_K(z)) \rho\lsub{K,z}(u)^{q-1} d\hm(z) du
		&=\frac{n+q-1}2 \int_{\sn}\int_K \rho\lsub{K,z}(u)^{q-1} dzdu \\
		&=\frac{(n+q-1)n}2 \int_K \wt V_{q-1}(K, z)\, dz \\
		&=\frac{(n+q-1)n\omega_n}{2q} I_q(K) \\
		&<\infty.
	\end{align*}
	This and \eqref{c2} give \eqref{c10}.
\end{proof}

An immediate consequence of Lemma \ref{I7} is the following lemma.

\begin{lemm}\label{I8}
	If $K\in\kn$ and $q>0$, then $\rho\lsub{K,z}(u)^{q-1}$ is integrable over
	\[
	D^+ =\{(z,u) \in \partial K \times \sn : \rho\lsub{K,z}(u)>0\},
	\]
	and $\wt V_{q-1}(K, \cdot)$ is integrable over
	$\partial K$.
\end{lemm}

\begin{proof}
	First, assume that the origin is in the interior of $K$. Then there is a constant $c>0$ so that
	$z\cdot \nu_K(z)>c$, for all $z\in\partial K$. Thus, by Lemma \ref{I7},
	\[
	c \int_{\partial K} \int_{S_z^+} \rho\lsub{K,z}(u)^{q-1} \, d\hm(z) du \le
	\int_{\partial K} \int_{S_z^+} \rho\lsub{K,z}(u)^{q-1} (z\cdot \nu_K(z)) \, d\hm(z) du <\infty.
	\]
	It follows that $\wt V_{q-1}(K, z)$ is integrable over
	$\partial K$ and
	\[
	\int_{D^+} \rho\lsub{K,z}(u)^{q-1} \, dz du
	= \int_{\partial K} \int_{S_z^+} \rho\lsub{K,z}(u)^{q-1} \, d\hm(z) du <\infty.
	\]
	Note that the above integrals are translation invariant.
\end{proof}

We are now able to obtain Theorem \ref{I9} for the case where the body is in $\kno$. This
will allow us to establish the property that  the centroid of each chord measure is the origin.

\begin{theo}\label{I9++} If $K\in \kno$  and $q>0$, then  the Borel measure $F_q(K,\cdot)$ is finite and
	\begin{equation}\label{c14++}
		I_q(K) = \frac1{n+q-1} \int_{\sn} h_K(v)\, dF_q(K,v) = \int_{\sn} dG_q(K,v).
	\end{equation}
\end{theo}

\begin{proof}
	From Lemma \ref{I8} we know that $\wt V_{q-1}(K, \cdot)$ is integrable over
	$\partial K$ when $q>0$. This and \eqref{c11} imply that $F_q(K,\cdot)$ is
	a finite Borel measure on $\sn$.
	Now \eqref{c14++} follows from \eqref{c10}, \eqref{c11} and \eqref{c11.1}.
\end{proof}

The following corollary says that the centroid of each chord measure is the origin
which is therefore known to be a necessary condition for a measure
to be the chord measure of some convex body.

\begin{coro} \label{I9.1}
	Let $K\in\kn$. For $q>0$, the chord measure $F_q(K,\cdot)$ satisfies
	\begin{equation}\label{c14.1}
		\int_{\sn} v\, dF_q(K,v) = 0.
	\end{equation}
\end{coro}

\begin{proof}
	Suppose firstly that
	$K\in\kno$. Let $u\in\sn$, and let $\lambda_0>0$ be such that
	$K+\lambda_0 u \in \kno.$
	By Theorem \ref{I9++},
	and the translation invariance of $F_q$,  we have
	\begin{align*}
		\int_{\sn} \lambda_0 u\cdot v\, dF_q(K,v)
		&= \int_{\sn} (h_K(v) +\lambda_0 u\cdot v -h_K(v))\, dF_q(K,v) \\
		&=\int_{\sn}  h_{K+\lambda_0 u}(v)\, dF_q(K+\lambda_0 u,v) - \int_{\sn} h_K(v)\, dF_q(K,v) \\
		&=(n+q-1)I_q(K+\lambda_0 u) - (n+q-1)I_q(K)\\
		&=0.
	\end{align*}
	Since $u\in \sn$ is arbitrary and $\lambda_0$ is positive, we get
	\[\int_{\sn} v\, dF_q(K,v) = 0.\]
	Since this holds for all $K\in \kno$,
	from the fact that $F_q(K,\cdot)$
	is invariant under translations of $K$,
	the conclusion follows for all $K\in\kn$.
\end{proof}

\vspace{3pt}

It is easy to see that for every $K\in\kn$, and every $q>0$,
the surface area measure, $S(K,\cdot)$, is absolutely continuous with respect to the chord measure $F_q(K,\cdot)$. This and the well-known fact (see, e.g. Schneider \cite{S14}) that the surface area measure of a body $K\in\kn$ cannot be concentrated in a closed hemisphere of $\sn$, shows that the same must be the case with the chord measure $F_q(K,\cdot)$. We combine this observation and Corollary \ref{I9.1} in one lemma.

\begin{lemm}\label{Izzz}
	Suppose $q>0$, and $\mu$ is a Borel measure on $\sn$. In order for there to exist a body $K\in\kn$ satisfying the equation
	\[
	F_q(K,\cdot) = \mu,
	\]
	it is necessary that $\mu$ not be concentrated on any hemisphere of $\sn$, and that
	\[
	\int_{\sn} u\,d\mu(u) = 0.
	\]
\end{lemm}

\begin{proof}[\bf Proof of Theorem \ref{I9}]
	Since by Remark 4.2 (d), we know that $F_q(K,\cdot)$ is invariant under translations of $K$, by Theorem \ref{I9++}, we see that $F_q(K,\cdot)$ is a finite Borel measure on $\sn$.
	
	Choose a $z_0\in K$ so that $-z_0 + K \in\kno$.  By  \eqref{c14++}, and the fact that $I_q(K)$ and $F_q(K,\cdot)$ are invariant under translations of $K$, we have
	\begin{align*}I_q(K) &= \frac1{n+q-1} \int_{\sn} h_{K-z_0}(v) dF_q(K,v) \\
		&= \frac1{n+q-1} \int_{\sn} h_K(v) dF_q(K,v)- \frac1{n+q-1} \int_{\sn} z_0\cdot v dF_q(K,v).\end{align*}
	Now, formula \eqref{c14} follows from the equation above and Corollary \ref{I9.1}. Formula \eqref{c14+} follows from \eqref{c14}   and \eqref{c11.1}.
\end{proof}

\vspace{5pt}

The following proposition contains the surprising connection between the dual quermassintegrals of a convex body and
the mean curvature of the boundary of the body.

\begin{prop}\label{I10}
	Suppose $K\in\kn$ and $q>0$. If $\partial K$ is second order differentiable
	at $z\in\partial K$, then
	\begin{equation}\label{c15}
		\lim_{q \to 0^+} q \wt V_{q-1}(K,z) = \frac{(n-1)\omega_{n-1}}{2n} H(z),
	\end{equation}
	where $H(z)$ is the mean curvature of $\partial K$ at $z$.
\end{prop}

\begin{proof}
	Let $\kappa_1,\ldots,\kappa_{n-1}$ denote the principal curvatures of $K$ at $z$.
	We now choose a coordinate system where our origin is at $z$ and where the basis vector $e_n=-\nu_K(z)$ and where the basis vectors $e_1,\ldots,e_{n-1}$ are eigenvectors of the Weingarten map of $K$ at $z$ corresponding to the eigenvalues
	$\kappa_1,\ldots,\kappa_{n-1}$. The basis vectors $e_1,\ldots,e_{n}$  are chosen to be orthogonal (since the Weingarten map is self-adjoint). It will be convenient to choose them to be orthonormal.
	
	Thus, in a neighborhood of $z$, the graph of $\partial K$
	may be written as
	\[
	y_n=f(y') =\frac12( \kappa_1
	{y'}_{\negthinspace\negthinspace 1}^2 + \cdots +
	\kappa_{n-1}
	{y'}_{\negthinspace\negthinspace n-1}^2) + o(|y'|^2),
	\]
	where ${y'}_{\negthinspace\negthinspace i} = y'\cdot e_i$.
	For $u\in\sn$, let $y= \rho\lsub{K,z}(u)u \in \partial K$, and let $dy$ be the surface area element of $\partial K$,
	while $dy'$ is the volume element in $\mathbb R^{n-1}$. Then, we have
	\[
	u=\frac1{\rho\lsub{K,z}(u)}(y', f), \ \ \ \nu_K(y)=\frac{(\nabla f, -1)}
	{(1+|\nabla f|^2)^\frac12},
	\ \ \ dy =  (1+|\nabla f|^2)^\frac12 dy'.
	\]
	With $du$ representing the surface area element of $\sn$,
	we also have
	\[
	\rho\lsub{K,z}(u)^n du = |y\cdot \nu_K(y)| dy, \ \ \ y'\negthinspace\cdot\negthinspace \nabla f = 2f(y') + o(|y'|^2).
	\]
	Thus,
	\[
	\rho\lsub{K,z}(u)^n du =|(y', f)\cdot (\nabla f, -1)| dy' =|y'\negthinspace\cdot\negthinspace \nabla f - f(y')| dy'
	= (f(y')+o(|y'|^2))dy'.
	\]
	Since $q \mapsto  \rho\lsub{K,z}(u)^{q-1}$ is bounded outside of a neighborhood of $z$, it is sufficient to restrict our attention to
	the subset $\omega(\delta) \subset S_z^+$ that corresponds to $|y'|\le \delta$, where $\delta>0$ is sufficiently small. Here,
	\begin{align*}
		&\lim_{q\to 0^+} q\int_{S_z^+} \rho\lsub{K,z}(u)^{q-1}\, du   \\=&
		\lim_{q\to 0^+} q\int_{\omega(\delta)} \rho\lsub{K,z}(u)^{q-1}\, du
		+ \lim_{q\to 0^+} q\int_{(S_z^+)\setminus\omega(\delta)} \rho\lsub{K,z}(u)^{q-1}\, du \\
	 =&\lim_{q\to 0^+} q\int_{\omega(\delta)} \rho\lsub{K,z}(u)^{q-1}\, du.
	\end{align*}
	
	Let $y'=rv'$, where $r\ge 0$ and $v'=({v'}_{\negthinspace\negthinspace 1},\ldots,{v'}_{\negthinspace\negthinspace {n-1}})\in S^{n-2}$, while
	the normal curvature  $\kappa(v')=\kappa_1 {v'}_{\negthinspace\negthinspace 1}^2 + \cdots + \kappa_{n-1} {v'}_{\negthinspace\negthinspace {n-1}}^2$ by Euler's theorem.
	We have
	
	\begin{align*}
		&q\int_{\omega(\delta)} \rho\lsub{K,z}(u)^{q-1}\, du\\
		=&q\int_{|y'|\le \delta} \rho\lsub{K,z}(u)^{q-n-1} (f(y')+o(|y'|^2))\, dy' \\
		=&q\int_{|y'|\le \delta} (|y'|^2+f^2)^\frac{q-n-1}2 (f(y')+o(|y'|^2))\, dy' \\
		=&q\int_{S^{n-2}}\int_0^\delta \big(r^2+ \frac14 r^4\kappa(v')^2+o(r^4)\big)^\frac{q-n-1}2
		\big(\frac12 r^2\kappa(v)+o(r^2)\big) r^{n-2}drdv' \\
		=&q\int_{S^{n-2}}\int_0^\delta r^{q-1}\big(1+\frac14 r^2\kappa(v')^2+o(r^2)\big)^\frac{q-n-1}2
		\big(\frac12 \kappa(v')+o(1)\big)\, drdv' \\
		=&q\int_{S^{n-2}}\int_0^\delta r^{q-1} \big(1+O(r^2)\big) \big(\frac12 \kappa(v')+o(1)\big)\, drdv'\\
		=&q\int_{S^{n-2}}\int_0^\delta r^{q-1}\big(\frac12 \kappa(v')+o(1)\big)\, drdv' \\
		=&\int_{S^{n-2}}\Big(\frac12 \delta^q \kappa(v') + q\int_0^\delta r^{q-1} o(1)dr\Big) dv',
	\end{align*}
	and thus,
	\[
	\lim_{q \to 0^+}
	q\int_{\omega(\delta)} \rho\lsub{K,z}(u)^{q-1}\, du
	\ =\
	\frac12 \int_{S^{n-2}} \kappa(v')\, dv'
	\ =\
	\frac12 H(z) (n-1)\omega_{n-1}.
	\]
\end{proof}

\begin{coro}
	Let $K\in\kn$. If $\partial K$ is $C^2$ with positive curvature,
	then
	\begin{equation}\label{I12}
		\lim_{q\to 0^+} F_q(K, \eta) = \frac{(n-1)\omega_{n-1}}{n\omega_n} S'(K, \eta),
	\end{equation}
	for each Borel set $\eta \subset \sn$.
\end{coro}

\begin{proof} Since $\partial K$ is $C^2$ with positive curvature, $H$ is continuous and
	\[
	S'(K,\eta) = \int_{\nu_K^{-1}(\eta)} H(z)\, d\hm(z).
	\]
	See Schneider \cite{S14}.
	From this, \eqref{c11}, and Proposition \ref{I10}, we get,
	\begin{align*}
		\lim_{q\to 0^+}F_q(K,\eta)
		&= \lim_{q\to 0^+}\frac{2q}{\omega_n} \int_{\nu_K^{-1}(\eta)} \wt V_{q-1}(K, z)\, d\hm(z) \\
		&= \frac{(n-1)\omega_{n-1}}{n\omega_n} \int_{\nu_K^{-1}(\eta)} H(z)\, d\hm(z) \\
		&=\frac{(n-1)\omega_{n-1}}{n\omega_n} S'(K, \eta).
	\end{align*}
\end{proof}


It would be very interesting to extend \eqref{I12} to all bodies in $\kn$.

\section{  The Differential Formula for Chord Integrals}

In this section, we will establish the differential formula for chord integrals
(which generalizes Aleksandrov's differential formula for volume) to all
positive-power chord integrals. This formula plays a critical role in solving
the chord Minkowski problem. One of the key ideas is the use of the $X$-ray function
instead of the extended radial function, this despite the fact that chord measures are naturally
defined via dual quermassintegrals (which are integrals of powers of the radial function).

For $K\in\kn$, a continuous function
$g : \sn \to \mathbb R$, and $\delta>0$, define $h_t:\sn\to\rbo$ by
\begin{equation}\label{v0}
	h_t(v) = h_K(v) + t g(v) + o(t,v),\qquad\text{$t\in (-\delta,\delta)$ and $v\in\sn$},
\end{equation}
where $o(t,\cdot)/t\rightarrow 0$ uniformly on $\sn$, as $t\to 0$.
The Wulff shape, $K_t$, of $h_t$ is
\begin{equation}\label{v1}
	K_t =\{x \in \rn : x\cdot v \le h_t(v) \text{ for all } v\in \sn\}\qquad t\in (-\delta,\delta).
\end{equation}
When $K$ contains the origin in its interior, the following differential formula
was established in \cite{HLYZ},
\begin{equation}\label{v2}
	\frac{d\rho\lsub{K_t}(u)}{dt}\Big|_{t=0} = \frac{g(\nu_K(\rhok(u)u))}{u\cdot \nu_K(\rhok(u)u)},
\end{equation}
for almost all $u\in\sn$.

The following is the differential formula for the extended radial function which is
a slight extension of \eqref{v2}.

\begin{lemm}\label{E1}
	Suppose $K\in\kn$, and $z$ is an interior point of $K$. If
	$K_t$ is the Wulff shape defined by \eqref{v1}, then for almost all $u\in \sn$,
	\begin{equation}\label{v3}
		\frac{d\rho\lsub{K_t,z}(u)}{dt}\Big|_{t=0}
		=
		\frac{g(\nu_K(z+\rho\lsub{K,z}(u)))}
		{u\cdot \nu_K(z+\rho\lsub{K,z}(u))}.
	\end{equation}
\end{lemm}

\begin{proof}
	Since
	\[
	\rho\lsub{K_t,z}(u) = \rho\lsub{K_t-z}(u),
	\]
	and
	\begin{align*}
		K_t-z &=\{x-z : x\in\rn\  \text{and}\  x\cdot v \le h_t(v), \text{ for all } v\in\sn\} \\
		&=\{y \in \rn : y\cdot v \le h_{K-z}(v) + t g(v) + o(t,v), \text{ for all } v\in\sn \} \\
		&=(K-z)_t.
	\end{align*}
	Thus
	\[
	\frac{d\rho\lsub{K_t,z}(u)}{dt}\Big|_{t=0} = \frac{d\rho\lsub{(K-z)_t}(u)}{dt}\Big|_{t=0}.
	\]
	Since $z$ is an interior point of $K$, the body $K-z$ contains the origin in its interior.
	By \eqref{v2}
	\[
	\frac{d\rho\lsub{(K-z)_t}(u)}{dt}\Big|_{t=0} = \frac{g(\nu_K(y))}{u\cdot \nu_K(y)},
	\]
	where
	$y=z+\rho\lsub{K,z}(u)u \in\partial K$.
	The desired formula \eqref{v3} follows.
\end{proof}

By using \eqref{v3}, we now derive the differential formula for the $X$-ray function.

\begin{lemm}\label{E1.1}
	Suppose $K\in\kn$ and
	$K_t$ is the Wulff shape defined by \eqref{v1}. If $u\in\sn$, then for almost all $x$ in the interior
	of $K|u^\perp$,
	\begin{equation}\label{v8}
		\frac{dX_{K_t}(x, u)}{dt}\Big|_{t=0}
		= \frac{g(\nu_K(y))}{u\cdot \nu_K(y)} - \frac{g(\nu_K(y^-))}{u\cdot \nu_K(y^-)},
	\end{equation}
	where $y$ and $y^-$ are the upper and lower points of $\partial K \cap (x+\mathbb Ru)$.
\end{lemm}

\begin{proof} Since $x$ is an interior of $K|u^\perp$, we can
	pick an interior point $z$ in $K$ so that
	\[
	K\cap (x+\mathbb Ru) = K\cap (z+\mathbb Ru).
	\]
	By \eqref{c6.2} and \eqref{v3}, we have
	\begin{align*}
		\frac{dX_{K_t}(x,u)}{dt}\Big|_{t=0} &= \frac{d\rho\lsub{K_t,z}(u)}{dt}\Big|_{t=0}
		+ \frac{d\rho\lsub{K_t,z}(-u)}{dt}\Big|_{t=0} \\
		&=\frac{g(\nu_K(y))}{u\cdot \nu_K(y)} - \frac{g(\nu_K(y^-))}{u\cdot \nu_K(y^-)}.
	\end{align*}
\end{proof}

The following lemma presents the identity for the integral forms of chord measures
in terms of dual quermassintegrals and the $X$-ray function.

\begin{lemm}\label{V1}
	Suppose $K\in\kn$, $q>-1$,  and  $g$ is a continuous function on $\sn$.  Then
	\[
	2n \int_{\partial K} \wt V_{q}(K,z) g(\nu_K(z))\, d\hm(z)
	= \int_{\sn} \int_{\partial K} X_K(z,u)^{q} g(\nu_K(z))\, d\hm(z) du.
	\]
\end{lemm}

\begin{proof}
	Since $g$ is continuous, $g\circ\nu_K$ is bounded and integrable over $\partial K$.
	For $u\in \sn$, let $\partial K(u)$ be the lower graph
	of $\partial K$ with respect $u$.
	Let
	\[
	D^+ =\{(z,u) \in \partial K \times \sn : \rho\lsub{K,z}(u)>0\}.
	\]
	It is known (see Schneider \cite{S14}) that
	$\partial K \cap (z+\mathbb Ru)$ consists of no more than two points for almost all $u \in \sn$.
	By appealing to Lemma \ref{I8}, we have
	\begin{align*}
		n \int_{\partial K} \wt V_{q}(K,z) g(\nu_K(z))\, d\hm(z)
		&= \int_{D^+} \rho\lsub{K,z}(u)^{q} g(\nu_K(z))\, d\hm(z) du \\
		&= \int_{\sn} \int_{\partial K(u)} \rho\lsub{K,z}(u)^{q} g(\nu_K(z))\, d\hm(z) du \\
		&= \int_{\sn} \int_{\partial K(-u)} \rho\lsub{K,z}(-u)^{q} g(\nu_K(z))\, d\hm(z) du,
	\end{align*}
	and
	\begin{multline*}
		\int_{\sn} \int_{\partial K} X_K(z,u)^{q} g(\nu_K(z))\, d\hm(z) du=\\
		\int_{\sn} \int_{\partial K(u)} \rho\lsub{K,z}(u)^{q} g(\nu_K(z))\, d\hm(z) du  \\
		+\int_{\sn} \int_{\partial K(-u)} \rho\lsub{K,z}(-u)^{q} g(\nu_K(z))\, d\hm(z) du.
	\end{multline*}
\end{proof}

A generalized dominated convergence theorem will be needed to establish the lemma to follow: Suppose $f_k, \phi_k, f, \phi$
are integrable functions in a measure space with $f_k \to f$, and $\phi_k \to \phi$, while
$|f_k| \le \phi_k$, almost everywhere.
If $\int \phi_k \to \int \phi$, then $\int f_k \to \int f$.
The following lemma is the crucial technical lemma needed in order to establish the differential formula for chord integrals.

\begin{lemm}\label{V2}
	Suppose $q>0$. If $K\in\kn$ and
	$K_t$ is the Wulff shape defined by \eqref{v1}, then
	there  is a class of nonnegative integrable functions $\phi_t(x, u)$
	defined for $u\in S^{n-1}$ and $x\in u^\bot$, such that
	\begin{equation}\label{E6-1}
		\Big|\frac1{t}\Big(X_{K_t}(x,u)^q - X_K(x,u)^q\Big)\Big| \leq \phi_t(x, u).
	\end{equation}
	Moreover, the limit function $\lim_{t\to 0}\phi_t(x,u)$ is integrable and
	\begin{equation}\label{E6-eq2}
		\lim_{t\to 0}\int_{\sn}\int_{u^\bot}\phi_t(x,u)\, dxdu =
		\int_{\sn}\int_{u^\bot} \lim_{t\to 0}\phi_t(x,u)\, dxdu.
	\end{equation}
\end{lemm}

\begin{proof}
	One can assume that the origin is inside the interior of $K$ because the
	$X$-ray $X_K(x,u)$ is invariant under translations of $K$.
	Since $g$ in \eqref{v1} is continuous, and since $o(t,\cdot)/t \rightarrow 0$
	uniformly on $S^{n-1}$,  there exist constants $c,\delta'>0$ so that
	\[ \Big|g(v) + \frac{o(t,v)}t\Big|\le c h_K(v), \quad
	{\rm for~all}\quad  v\in \sn, t\in(-\delta',\delta').\]
	Then
	\[
	(1-c|t|)K \subset K_t \subset (1+c|t|)K.
	\]
	Thus,
	\begin{equation}\label{eq-ctrl1}
		\Big|\frac1{t}\big(X_{K_t}(x,u)^q - X_K(x,u)^q\big)\Big|  \leq
		\phi_t(x,u),
	\end{equation}
	where
	\begin{equation}
		\phi_t(x,u)=
		\frac1{|t|}
		\big( X_{(1+c|t|)K}(x,u)^q-  X_{(1-c|t|)K}(x,u)^q\big).
	\end{equation}
	
	Thus \eqref{E6-1} holds.
	
	The homogeneity of $I_q(K,u)$, gives us
	\begin{align*}
		\int_{u^\bot} \phi_t(x,u)\, dx
		&= \frac1{|t|} \big(I_q((1+c|t|)K,u) - I_q((1-c|t|)K, u)\big)\\
		&= \frac1{|t|} \big( (1+c|t|)^{n+q-1}- (1-c|t|)^{n+q-1}\big)I_q(K,u).
	\end{align*}
	Therefore,
	\[
	\int_{S^{n-1}}\int_{u^\bot} \phi_t(x, u)\, dxdu
	= \frac1{|t|}\big((1+c|t|)^{n+q-1}- (1-c|t|)^{n+q-1}\big) n\omega_n I_q(K),
	\]
	and thus
	\[ \lim_{t\to 0}\int_{\sn}\int_{u^\bot} \phi_t(x,u)\, dxdu  = 2c(n+q-1)n\omega_nI_q(K).\]
	
	On the other hand, by \eqref{v8}, when $x$ is an interior point of $K|u^\perp$, we have
	\[\lim_{t\rightarrow0} \phi_t (x,u) = 2qc X_K(x,u)^{q-1}
	\Big(\frac{h_K(\nu_K(y))}{u\cdot\nu_K(y)}-\frac{h_K(\nu_K(y^-))}{u\cdot\nu_K(y^-)}\Big),\]
	where $y$ and $y^-$ are the two boundary points of
	$\partial K\cap (x+\mathbb Ru)$.
	Since for almost all $u\in\sn$, $\partial K\cap (x+\mathbb Ru)$ consists of at most two points,
	we get
	\begin{align*}&\int_{S^{n-1}}\int_{u^\bot} \lim_{t\to 0}\phi_t(x,u)\, dxdu\\
		=&2qc \int_{S^{n-1}}\int_{K|u^\perp} X_K(x,u)^{q-1}
		\Big(\frac{h_K(\nu_K(y))}{u\cdot\nu_K(y)}-\frac{h_K(\nu_K(y^-))}{u\cdot\nu_K(y^-)}\Big)\,dxdu\\
		 =& 2qc\int_{S^{n-1}} \int_{\partial K}  X_K(y,u)^{q-1} h_K(\nu_K(y))\,d\hm(y)du.
	\end{align*}
	By Lemmas \ref{I7} and \ref{V1}, we obtain
	\[I_q(K) = \frac{q}{(n+q-1)n\omega_n}\int_{S^{n-1}} \int_{\partial K}
	X_K(y,u)^{q-1} h_K(\nu_K(y))\,d\hm(y)du. \]
	Therefore, we obtain
	\[\int_{S^{n-1}}\int_{u^\bot} \lim_{t\rightarrow 0}\phi_t(x,u)\, dxdu
	= 2c(n+q-1)n\omega_nI_q(K).\]
	Thus both sides of equation \eqref{E6-eq2} are equal to $2c(n+q-1)n\omega_nI_q(K)$.
\end{proof}

The following is the variational formula for chord integrals that shows that
chord measures are differentials of chord integrals.

\begin{theo}\label{V3}
	Let $K\in\kn$ and $q>0$.
	Suppose that $g : \sn \to \mathbb R$ is a continuous function
	and $h_t:\sn\to\rbo$ be given by
	\begin{equation*}
		h_t(v) = h_K(v) + t g(v) + o(t,v),\qquad v\in \sn, \ t\in (-\delta,\delta), \ \delta>0,
	\end{equation*}
	where $o(t,\cdot)/t\rightarrow 0$ uniformly on $\sn$, as $t\to 0$.
	If
	\begin{equation*}
		K_t =\{x \in \rn : x\cdot v \le h_t(v) \text{ for all } v\in \sn\},    \qquad t\in (-\delta,\delta),
	\end{equation*}
	is the Wulff shape of $h_t$, then
	\begin{equation}\label{V3-1}
		\frac d{dt}\Big|_{t=0} I_q(K_t) =
		\frac {2q}{\omega_n} \int_{\partial K}  \wt V_{q-1}(K,z) g(\nu_K(z))\, d\hm(z)
		= \int_{\sn} g(v)\, dF_q(K,v).
	\end{equation}
\end{theo}

\begin{proof}
	By Lemma \ref{V2}, there is a family of nonnegative integrable function $\phi_t(x,u)$
	satisfying \eqref{E6-1} and \eqref{E6-eq2}.
	Then by the generalized dominated convergence theorem and Lemmas \ref{V2},
	\ref{E1.1} and \ref{V1}, we have
	\begin{align*}
		\frac d{dt}\Big|_{t=0} I_q(K_t) &=  \frac1{n\omega_n}\lim_{t\to 0}
		\int_{S^{n-1}}\int_{u^\perp} \frac1t \Big(X_{K_t}(x,u)^q -  X_K(x,u)^q\Big)\, dxdu \\
		&=\frac1{n\omega_n}\int_{S^{n-1}}\int_{u^\bot} \lim_{t\to 0} \frac1t
		\Big(X_{K_t}(x,u)^q-X_K(x,u)^q\Big)\,dx du \\
		&=\frac q{n\omega_n}\int_{S^{n-1}}\int_{u^\bot} X_K(x,u)^{q-1}
		\Big(\frac{g(\nu_K(y))}{u\cdot\nu_K(y)}-\frac{g(\nu_K(y^-))}{u\cdot\nu_K(y^-)}\Big)\,dxdu\\
		&= \frac q{n\omega_n}\int_{S^{n-1}} \int_{\partial K}   X_K(y,u)^{q-1} g(\nu_K(y))\, d\hm(y)du \\
		&=\frac {2q}{\omega_n}\int_{\partial K}  \wt V_{q-1}(K,z) g(\nu_K(z))\, d\hm(z),
	\end{align*}
	where $y$ and $y^-$ are the two boundary points of
	$\partial K\cap (x+\mathbb Ru)$.
\end{proof}

Note that the special case of $q=1$ of the variational formula \eqref{V3-1} is the volume
variational formula of Aleksandrov.

\begin{lemm}\label{V3+}
	Suppose $K\in \kno$ and $q>0$. Suppose further that $g :\sn \to \rbo$ is a
	continuous function, while $K_t$ is the Wulff shape of
	$h_t$ defined for $|t|<\delta$, for some $\delta>0$ by
	\[
	\log h_t = \log h_K + t g + o(t, \cdot),\qquad\text{on $\sn$}
	\]
	where $o(t,\cdot)/t \to 0$ uniformly on $\sn$ as $t\to 0$.
	Then the differential of $I_q(K_t)$, is given by
	\begin{equation}\label{V3.1}
		\frac d{dt}\Big|_{t=0}I_q(K_t) = (n+q-1) \int_{\sn}g(v) dG_q(K,v).
	\end{equation}
\end{lemm}

\begin{proof} It is easily seen that
	\[
	h_t = h_K + t gh_K + o(t, \cdot),\qquad\text{on $\sn$}
	\]
	Then \eqref{V3.1} follows immediately from \eqref{V3-1}.
\end{proof}

\section{  The Minkowski Problem for Chord Measures}

In this section,
we solve the Minkowski problem for the chord measure $F_q(K,\cdot)$.

\vspace{5pt}


\vspace{5pt}

For a finite Borel measure $\mu$  on $\sn$, and $f\in C^+(\sn)$, the positive continuous function on $\sn$, let
\[
\|f\negthinspace\negthinspace:\negthinspace\mu\| = \int_{\sn} f(u)\, d\mu(u).
\]
Define the functional $\Phi_{q,\mu}:\kno\to \rbo$, by
\[
\Phi_{q,\mu}(K) = I_q(K)^\frac1{n+q-1} / \|h_K\negthinspace\negthinspace:\negthinspace\mu\|.
\]

The following lemma reduces the chord Minkowski problem to a maximization of a functional
on convex bodies.

\begin{lemm}\label{M1}
	Suppose $q>0$, and $\mu$ is a nonzero finite Borel measure on $\sn$. If the maximization problem
	\[
	\sup\{\Phi_{q,\mu}(K) : K\in \kno\}
	\]
	has a solution in $\kno$, then there exists a body $K\in\kno$ so that
	\[
	F_q(K,\cdot) = \mu.
	\]
\end{lemm}

\begin{proof}
	By assumption, the maximization problem has a solution $K_0\in\kno$. The origin being in the interior of $K_0$ implies
	$h_{K_0}>0$.
	
	Suppose $g\in\csn$. Then
	$h_t= h_{K_0} + t g >0$, whenever $|t|<\delta$, for some sufficiently small $\delta >0$.
	Obviously, the function $t\mapsto\Phi_{q,\mu}(K_t)$ attains its maximum at $t=0$.
	Note that $t\mapsto\Phi_{q,\mu}(K_t)$ may not be differentiable
	at $t=0$ (because $t\mapsto\|h_{K_t}\negthinspace:\negthinspace\mu\|$ may not be differentiable at $t=0$). Let
	\[
	\phi(t) = I_q(K_t)^\frac1{n+q-1} / \|h_t\negthinspace:\negthinspace\mu\|.
	\]
	Since the body $K_t$ is the Wulff shape of the function $h_t$ and the measure $\mu$ is non-negative it follows that
	$\|h_{K_t}\negthinspace:\negthinspace\mu\| \le
	\|h_t\negthinspace:\negthinspace\mu\|$, and thus that
	\[
	\phi(t) \le \Phi_{q,\mu}(K_t) \le \Phi_{q,\mu}(K_0) = \phi(0).
	\]
	Thus, $\phi(t)$ attains maximum at $t=0$.
	By Theorem \ref{V3}, the function  $t\mapsto I_q(K_t)$ is differentiable at $t=0$. Of course,
	$t\mapsto\|h_t\negthinspace:\negthinspace\mu\|$ is obviously
	differentiable at $t=0$. Therefore,
	\begin{align*}
		0=&\frac d{d{t}}\phi(t)\Big|_{t=0}\\
		=&\frac{I_q(K_0)^{\frac1{n+q-1}-1}}{(n+q-1) \|h_{K_0}\negthinspace:\negthinspace\mu\|} \int_{\sn} g(v)\, dF_q(K_0,v)
		- \frac{I_q(K_0)^\frac1{n+q-1}}{ \|h_{K_0}\negthinspace:\negthinspace\mu\|^{2}} \int_{\sn} g(v)\, d\mu(v).
	\end{align*}
	Since this hold for all $g\in\csn$, we conclude that
	\[
	\frac1{n+q-1} \frac{\|h_{K_0}\negthinspace:\negthinspace\mu\|}{I_q(K_0)} F_q(K_0, \cdot) = \mu.
	\]
	Now, letting $K=\lambda K_0$ with $\lambda^{n+q-2} = \frac1{n+q-1} \frac{\|h_{K_0}:\mu\|}{I_q(K_0)}$
	gives the desired body.
\end{proof}

For $q>0$, the existence of a solution to the chord Minkowski problem is contained in the following theorem.

\begin{theo}\label{M2}
	Suppose $q>0$.
	If $\mu$ is a non-zero, finite Borel measure on $\sn$, then there exists a convex body $K\in\kn$ so that
	\[
	F_q(K,\cdot) = \mu
	\]
	if and only if
	$\mu$
	is not concentrated on any
	closed hemisphere of $\sn$, and
	\begin{equation}\label{*}
		\int_{\sn} v\, d\mu(v) =0.
	\end{equation}
	
\end{theo}

\begin{proof}
	The necessity of the condition is established in Lemma \ref{Izzz}. We proceed to establish  sufficiency.
	We shall demonstrate that
	the maximization problem
	\[
	\sup\{\Phi_{q,\mu}(K) : K\in \kno\}
	\]
	has a solution in $\kno$.
	To that end, suppose $K_i$ is a maximizing sequence for the functional $\Phi_{q,\mu}:\kno\to\rbo$.
	
	Since the measure $\mu$ satisfies the condition \eqref{*}, it follows that $K\mapsto \|h_K\negthinspace\negthinspace:\negthinspace\mu\|$ is a translation invariant function.
	Since $I_q$ is translation invariant, we now see that $\Phi_{q,\mu}$ is translation invariant as well.
	We therefor may assume that the centroid of each of the $K_i$ is at the origin.

	Since $\Phi_{q,\mu}$ is homogeneous of degree $0$,  we may rescale the bodies in our maximizing
	sequence so that $I_q(K_i)=1$, for all $i$.
	So, $K_i$ is a minimizing sequence for the problem,
	\[
	\inf \{\|h_K\negthinspace\negthinspace:\negthinspace\mu\| : I_q(K)=1,\ K\in \kno\}.
	\]
	
	Since $K_i$ is a minimization sequence, we can assume that for the ball $B'$ centered
	at the origin and whose radius $r'$ is such that $I_q(B')=1$,
	we have
	\[
	\|h_{K_i}\negthinspace\negthinspace:\negthinspace\mu\| \le
	\|h_{B'}\negthinspace\negthinspace:\negthinspace\mu\|  = r' |\mu|.
	\]
	where $|\mu| = \mu(\sn)$.
	
	Let $R_i>0$ and $u_i\in \sn$ be such that
	\[
	R_i = \max_{u\in \sn} \rho\lsub{K_i}(u) = \rho\lsub{K_i}(u_i).
	\]
	Then the line segment joining the origin and the point $R_i u_i$ is contained inside of $K_i$.
	This line segment obviously has
	support function $v \mapsto R_i (u_i\cdot v)_+$, and thus,
	\[
	R_i (u_i\cdot v)_+ \le h_{K_i}(v), \ \ \text{ for all } v\in \sn,
	\]
	which implies
	\[
	R_i \int_{\sn} (u_i\cdot v)_+ \, d\mu(v) \le \|h_{K_i}\negthinspace\negthinspace:
	\negthinspace\mu\| \le r' |\mu|.
	\]
	Since $\mu$ is not concentrated on any hemisphere, the function
	\[
	u \mapsto\int_{S^{n-1}} (u\cdot v)_+ d\mu(v)
	\]
	is strictly positive on $S^{n-1}$, and since the function is continuous on compact $S^{n-1}$
	there is a $c>0,$ such that
	\[
	\int_{S^{n-1}} (u\cdot v)_+ d\mu(v) \geq c  > 0\qquad\text{for all $u\in S^{n-1}$.}
	\]
	Thus we have our uniform bound for $R_i$,
	\[
	R_i \le {r'|\mu|}/c.
	\]
	Since the sequence $K_i$ is bounded,
	the Blaschke selection theorem provides a convergent subsequence, which we again denote as $K_{i}$,
	that converges to a compact convex set  $K_0\in\cn$.
	The continuity of $I_{q}$ assures that $I_{q}(K_0) = 1$. But $q>0$, so
	$K_0$ must have nonempty interior as well as having its centroid at the origin.
	Thus, $K_0 \in \kno$.
	Since $K_i$ is a maximizing sequence for $\Phi_{q,\mu}$, the body $K_0$ is a maximizer of
	$\Phi_{q,\mu}$, and Lemma \ref{M1} now gives the desired result.
\end{proof}

\section{  Estimates for Entropy and Chord Integrals}

We shall use a variational method to solve the cone chord Minkowski problem.
The associated optimization problem  involves chord integrals and the
entropy of a convex body relative to a given measure. To solve the optimization problem,
delicate estimates for chord integrals and entropy are needed. This section
contains such estimates.

Let $\mu$ be a Borel measure on $\sn$. The {\it entropy} functional $\eop_\mu:\kno\to\rbo$
is defined for $K\in \kno$,
\[
\eop_\mu(K) = -\frac1{|\mu|}\int_{\sn} \log h_K\, d\mu,
\]
and $\eop_\mu(K)$ is called the entropy
of $K$ with respect to $\mu$. The functional is obviously monotone decreasing with respect to set inclusion. It is obviously continuous.

Assume that the measure $\mu$ satisfies the {\it subspace mass inequality} for index $1<q<n+1$,
\begin{equation} \label{mass1+}
	\frac{\mu(\xi_k\cap S^{n-1})}{|\mu|}  <  \frac{k+\min\{k,q-1\}}{n+q-1},
\end{equation}
for each $\xi_k\in G_{n,k}$, and for each $k=1, \ldots, n-1$.

Inequality \eqref{mass1+} is weaker than the inequality $\mu(\xi_k\cap S^{n-1})\leq \frac kn |\mu|$
satisfied by cone volume measure. A different measure concentration inequality is
satisfied by dual curvature measurers. Estimates for entropy $\eop_\mu(K)$ have been derived
based on measure concentration properties of cone volume measure and dual curvature
measures, see for example, \cite{BLYZ13jams, HLYZ, Zhao17jdg, BLYZZ1}.
Here we derive new estimates for entropy $\eop_\mu(K)$ based on
the concentration properties of chord measures.
The crucial technique for estimating $\eop_\mu(K)$ is a spherical partition which deals with
measure concentration on subspheres, see \cite{BLYZ13jams}, and \cite{HLYZ}
for a somewhat different formulation.

Let $e_1,\cdots, e_n$ be an orthonormal basis in $\mathbb{R}^n$.
For each $\delta\in (0,\frac{1}{\sqrt{n}})$, define
a partition $\{\Omega_{1,\delta}, \ldots, \Omega_{n,\delta}\}$ of $\sn$ by
\begin{equation}\label{pt}
	\Omega_{k,\delta}=\{v\in S^{n-1}: |v\cdot e_k|\geq \delta \text{ and }|v\cdot e_j|<\delta
	\text{ for all } j>k\}, \ \ k=1,\cdots, n.
\end{equation}

Let $\xi_0 = \{0\}$, and
\begin{equation*}
	\xi_k = \Span\{e_1,\cdots, e_k\}, \quad k=1,\cdots, n.
\end{equation*}

It was shown in \cite{BLYZ13jams} that
for any non-zero finite Borel measure $\mu$ on $S^{n-1}$,
\begin{equation}\label{pt1}
	\lim_{\delta\rightarrow 0^+}\mu(\Omega_{k,\delta})
	=\mu((\xi_k\setminus \xi_{k-1})\cap S^{n-1}),
\end{equation}
and, therefore,
\begin{equation}\label{pt2}
	\lim_{\delta\rightarrow 0^+}\big(\mu(\Omega_{1,\delta}) + \cdots +
	\mu(\Omega_{k,\delta})\big)=\mu(\xi_k\cap S^{n-1}).
\end{equation}

The following elementary lemma from \cite{BLYZZ1} will be used.

\begin{lemm} \cite[Lemma 4.1]{BLYZZ1} \label{lm4.1}
	Suppose $\lambda_1,...,\lambda_m\in[0,1]$ are such that
	\begin{equation*}
		\lambda_1+\cdots +\lambda_m = 1.
	\end{equation*}
	Suppose further that $a_1\le  \cdots \le a_m$ are real numbers.
	Assume there exists $\sigma_0,
	\ldots,
	\sigma_m$
	$\in[0,\infty)$, with
	$\sigma_0=0$, and  $\sigma_m=1$, such that
	\begin{equation*}
		\lambda_1+\cdots+\lambda_k\le \sigma_k, \quad {\rm for}~k=1,...,m.
	\end{equation*}
	Then
	\begin{equation*}
		\sum_{k=1}^m\lambda_ka_k\ge \sum_{k=1}^m(\sigma_k-\sigma_{k-1})a_k.
	\end{equation*}
\end{lemm}

The following lemma gives the entropy estimate that is needed for solving
the chord log-Minkowski problem.

\begin{lemm}\label{e-0}
	Suppose  $1\le q \le n+1$.
	Suppose $\mu$ is a finite Borel measure on $\sn$
	satisfying the subspace mass inequality
	\begin{equation*}
		\frac{\mu(\xi_k\cap S^{n-1})}{|\mu|}  <  \frac{k+\min\{k,q-1\}}{n+q-1},
	\end{equation*}
	for each $\xi_k\in G_{n,k}$, and for each $k=1, \ldots, n-1$.
	Suppose further that a sequence $(e_{1l},..., e_{nl})$, for
	$l=1, 2,...$, of ordered orthonormal bases of $\rn$, converges to
	the ordered orthonormal basis $(e_1,..., e_n)$.
	If $\varepsilon_0>0$, and $E_l$ is the sequence of ellipsoids
	\[
	E_l = \Big\{ x \in \rn: \frac{(x\cdot e_{1l})^2}{a_{1,l}^2}
	+ \cdots + \frac{(x\cdot e_{nl})^2}{a_{n,l}^2} \leq 1\Big\},
	\]
	with
	\[
	\text{$0<a_{1,l}\le a_{2,l} \le \cdots \le a_{n,l}$,
		\quad and\ \  $a_{n,l}\ge \varepsilon_0$},
	\]
	for all $l$, then there exist $l_0$, $t_0>0$ and $c_0\in\rbo$ such that for each $l>l_0$,
	\begin{multline}\label{e-1}
		\eop_\mu(E_l) \leq -\frac{2 \log (a_{1,l}\cdots a_{\lfloor{q}\rfloor-1, l})}{n+q-1}
		- \frac{q-\lfloor q\rfloor+1}{n+q-1} \log  a_{\lfloor q \rfloor,l}
		-\frac{\log (a_{\lfloor q \rfloor+1, l} \cdots a_{n,l})}{n+q-1} \\ + t_0\log a_{1,l} + c_0.
	\end{multline}
\end{lemm}

The first term of \eqref{e-1} does not appear if $\lfloor q\rfloor=1$,
the second term of \eqref{e-1} does not appear if $q=n+1$,
and the third term of \eqref{e-1} does not appear if $\lfloor q\rfloor\ge n$.

\begin{proof}
	For each $\delta\in (0,1/\sqrt n)$, let
	$\{\Omega_{1,\delta},\ldots, \Omega_{n,\delta}\}$ be the partition of $S^{n-1}$
	defined by \eqref{pt}.
	Denote  $\lambda_{k,\delta} = \mu (\Omega_{k,\delta})/|\mu|$.
	Then, by \eqref{pt2} and \eqref{mass1+}, we have
	\begin{equation*}
		\lim_{\delta\rightarrow0^+} (\lambda_{1,\delta}+\cdots+\lambda_{k,\delta})
		= \frac{\mu(\xi_k\cap S^{n-1})}{|\mu|}< \frac{k +\min\{k, q-1\}}{n+q-1}.
	\end{equation*}
	Choose $t_0,\delta_0>0$, such that for each $k=1,...,n-1,$
	\begin{equation*}
		\lambda_{1,\delta_0}+\cdots+\lambda_{k,\delta_0}
		< \frac{k +\min\{k, q-1\}}{n+q-1} - t_0.
	\end{equation*}
	Let $\sigma_n=1$ and $\sigma_0=0$ and $\sigma_k = \frac{k +\min\{k, q-1\}}{n+q-1} - t_0$, for $k=1, \ldots, n-1$.
	Then
	\begin{align}
		\sigma_1 -\sigma_0 &= \frac{\min\{2,q\}}{n+q-1} - t_0, \label{s-df1}\\
		\sigma_k -\sigma_{k-1}&=\begin{cases} \frac2{n+q-1} &\text{if }k<\lfloor q\rfloor \\
			\frac{q-\lfloor q\rfloor+1}{n+q-1} &\text{if }k=\lfloor q\rfloor \\
			\frac1{n+q-1} &\text{if }k>\lfloor q\rfloor \end{cases},  \ \  k=2,\ldots, n-1, \label{s-df}\\
		\sigma_n -\sigma_{n-1}&=\frac{1-\min\{n-q, 0\}}{n+q-1}+t_0.  \label{s-df2}
	\end{align}
	
	Since $(e_{1l},...,e_{nl})$ converges to $(e_1,...,e_n)$, we may choose $l_0>0$,
	so that $|e_{kl}-e_k| < \delta_0/2$ for each $l>l_0$ and each $k\in \{1,...,n\}$.
	Since $a_{k,l}e_{kl} \in E_l$,
	for $v\in\Omega_{k,\delta_0}$  we have
	\[
	h_{E_l}(v) \geq  |v\cdot (a_{k,l} e_{kl})| \geq a_{k,l}(|v\cdot  e_{k}|-|e_{kl}-e_{k}|)
	\geq a_{k,l}\frac{\delta_0}2.
	\]
	This  and Lemma \ref{lm4.1}, allow us to deduce that for $l>l_0$,
	\begin{align*}
		\eop_\mu(E_l) &= -\frac1{|\mu|} \int_{\sn} \log h_{E_l}\, d\mu \\
		&\leq -\sum_{k=1}^n  \frac{\mu(\Omega_{k,\delta_0})}{|\mu|} \log \Big(a_{k,l}\frac{\delta_0}2\Big)  \\
		&\leq -\log\frac{\delta_0}2 - \sum_{k=1}^n \lambda_{k,\delta_0} \log a_{k,l} \\
		&\leq -\log\frac{\delta_0}2 - \sum_{k=1}^n (\sigma_k - \sigma_{k-1})\log a_{k,l}.
	\end{align*}
	Consequently, by using \eqref{s-df1}, \eqref{s-df} and \eqref{s-df2},  we have
	\begin{multline}\label{s-df3}
		\eop_\mu(E_l) \leq -\log\frac{\delta_0}2 -\frac{\min\{2,q\}}{n+q-1}\log a_{1,l} + t_0\log a_{1,l}
		-\frac2{n+q-1} \sum_{k=2}^{\lfloor q\rfloor-1}\log a_{k,l}  \\
		-  \frac{q-\lfloor q\rfloor+1}{n+q-1}\log a_{\lfloor q\rfloor,l}
		- \frac1{n+q-1}\sum_{k=\lfloor q\rfloor +1}^{n-1} \log a_{k,l}
		-\frac{1-\min\{n-q,0\}}{n+q-1} \log a_{n,l} \\ - t_0\log a_{n,l}, \end{multline}
	where obviously each of the two sums does not appear if its lower limit is larger than its upper limit,
	and the term containing $a_{\lfloor q\rfloor,l}$ does not appear if $q \notin [2, n)$.
	Note that since obviously
	\[
	q-\lfloor q\rfloor +1 = \begin{cases} \min\{2, q\}  &\text{if } \lfloor q\rfloor =1 \\
		1-\min\{n-q, 0\}  &\text{if } \lfloor q\rfloor=n, \end{cases}
	\]
	we can rewrite \eqref{s-df3} as
	\begin{multline}\notag
		\eop_\mu(E_l) \leq  -\frac2{n+q-1} \sum_{k=1}^{\lfloor q\rfloor-1}\log a_{k,l}
		-  \frac{q-\lfloor q\rfloor+1}{n+q-1}\log a_{\lfloor q\rfloor,l}
		 \\ - \frac1{n+q-1}\sum_{k=\lfloor q\rfloor +1}^{n} \log a_{k,l}
		-\log\frac{\delta_0}2  + t_0\log a_{1,l}   - t_0\log a_{n,l}, \end{multline}
	where again each of the two sums does not appear if its lower limit is larger than its upper limit,
	and the term containing $a_{\lfloor q\rfloor,l}$ does not appear if $q=n+1$.
	This and the assumption that  $a_{nl}\geq \varepsilon_0$ give the desired result.
\end{proof}

The following lemma gives an estimate for the chord integral $I_{q}(E)$
when $q\in (1,n+1)$ is a non-integer and $E$ is an ellipsoid.

\begin{lemm} \label{e-2}
	Suppose $q\in (1,n+1)$ is not an integer.
	If $E$ is the ellipsoid in $\rn$ given by
	\begin{equation*}
		E=E(a_1,\ldots,a_n) = \Big\{ x \in \rn: \frac{(x\cdot e_{1})^2}{a_{1}^2} + \cdots
		+ \frac{(x\cdot e_{n})^2}{a_{n}^2} \leq 1 \Big\}
	\end{equation*}
	with $0<a_{1}\le a_{2} \le \cdots \le a_{n}\leq 1$, then for any real $q$ and integer $m$
	such that $1\le m < q < m+1 \le n+1$,
	\[
	I_{q}(E)\leq c_{q,m,n}  (a_1\cdots a_m)^2a_m^{q-m-1}  a_{m+1}\cdots a_n,
	\]
	where $c_{q,m,n}$ is a constant that depends only on $q$ and $n$ (since $m=\lfloor q \rfloor$) and is given by
	\begin{equation}
		c_{q,m,n} = \begin{cases} \dfrac{2^{q-n+2}q(q-1)\omega_{n-1}^2}{(q-n)(q-n+1)n\omega_n}   &m=n \\
			\dfrac{2^{n-m+3}q(q-1)(n-m)\omega_{m-1}^2\omega_{n-m}^2}{(m+1-q)(q-m)(q-m+1)n\omega_n}  &m< n.
		\end{cases} \label{cqmp}
	\end{equation}
\end{lemm}

\begin{proof}
	Let $D$ and $J$ be the ellipsoids in $\R^{m-1}$ and $\R^{n-m+1}$ defined by
	\begin{align*}
		D&=\Big\{ x \in \R^{m-1}: \frac{(x\cdot e_{1})^2}{a_{1}^2} + \cdots
		+ \frac{(x\cdot e_{m-1})^2}{a_{m-1}^2} \leq 1\Big\}, \\
		J&=\Big\{ x \in \R^{n-m+1}:  \frac{(x\cdot e_{m})^2}{a_{m}^2} + \cdots
		+ \frac{(x\cdot e_{n})^2}{a_{n}^2} \leq 1\Big\},
	\end{align*}
	where letting $D= \{0\}$ if $m=1$.
	
	Since $E$ is contained in   $D\times J$, obviously $I_{q}(E) \leq I_{q}(D\times J)$.
	
	If $(x_1,y_1), (x_2, y_2) \in \R^{m-1}\times \R^{n-m+1}$, it follows that
	$|x_1-x_2+y_1-y_2|\ge |y_1-y_2|$,
	since $x_1-x_2$ and $y_1-y_2$ are in complementary subspaces. Therefore, using \eqref{c6.3.1}, we have
	\begin{align}
		I_{q}(E) &\leq \frac{q(q-1)}{n\omega_n} \int_{x_1,x_2\in D}\int_{y_1,y_2\in J}
		|x_1-x_2+y_1-y_2|^{q-1-n} dy_1dy_2dx_1dx_2 \notag\\
		&\leq \frac{q(q-1)}{n\omega_n}
		\int_{x_1,x_2\in D}\int_{y_1,y_2\in J}
		|y_1-y_2|^{q-1-n}  dy_1dy_2dx_1dx_2\notag\\
		&= \frac{q(q-1)}{n\omega_n} \vv_{m-1}(D)^2 \int_{y_1,y_2\in J} |y_1-y_2|^{q-1-n} dy_1dy_2, \notag\\
		&= \frac{q(q-1)\omega_{m-1}^2}{n\omega_n}(a_1\cdots a_{m-1})^2
		\int_{y_1,y_2\in J}|y_1-y_2|^{q-1-n} dy_1dy_2,     \label{e-Iq1}
	\end{align}
	where $\vv_{m-1}(D)$ is the $(m-1)$-dimensional volume of $D\subset \R^{m-1}$ with $\vv_{0}(D)=1$, if $m=1$.
	
	Next we estimate the last integral, which is the chord power integral of $J$
	in $\R^{n-m+1}$. We consider the two cases $m=n$ and $m\le n-1$ separately.
	\vskip 5pt
	\noindent{\it Suppose $m = n$.} Then $J=[-a_n,a_n]$. An elementary calculation gives that
	\begin{align}
		\int_{y_1, y_2\in J} |y_1-y_2|^{q-1-n} dy_1dy_2
		=&\int_{-a_n}^{a_n}\int_{-a_n}^{a_n} |y_1-y_2|^{q-1-n} dy_1dy_2 \notag \\
		=& \frac{1}{q-n}  \int_{-a_n}^{a_n}    \big(  (a_n+y_2)^{q-n} + (a_n-y_2)^{q-n}\big)dy_2\notag\\
		=&\frac{2^{q-n+2}}{(q-n)(q-n+1)}a_n^{q-n+1}.   \label{e-Iq2}
	\end{align}
	\vskip 5pt
	
	\noindent{\it Assume $m \le n-1$.}
	Note that $J \subset N\times M$, where
	\[
	N=[-a_m,a_m], \ \ \ M =\Big\{ x \in \R^{n-m}:  \frac{(x\cdot e_{m+1})^2}{a_{m+1}^2} + \cdots
	+ \frac{(x\cdot e_{n})^2}{a_{n}^2} \leq 1\Big\}.
	\]
	Since $0<a_{1}\le a_{2} \le \cdots \le a_{n}\leq 1$, we conclude that
	\[
	M-z\subset 2B^{n-m}, \ \ \ \text{for any }\ z\in M.
	\]
	Note that for $(t_1,z_1), (t_2, z_2) \in \R\times \R^{n-m}$,
	\[
	|(t_1-t_2,z_1-z_2)| = \Big(|t_1-t_2|^2 + |z_1-z_2|^2 \Big)^{1/2}
	\geq   (|t_1-t_2| + |z_1-z_2|)/2.
	\]
	The last observation, a change of variable $z'=z_1-z_2$, the fact that  $M-z_2\subset 2B^{n-m}$,
	switching to polar coordinates, the obvious $|t_1-t_2|+r\ge r$ and $n-m-1\geq 0$,  give
	\begin{align} &\int_{y_1, y_2\in J} |y_1-y_2|^{q-1-n} dy_1dy_2 \notag \\
		=&  \int_{t_1, t_2\in N} \int_{z_1, z_2\in M}|(t_1-t_2,z_1-z_2)|^{q-1-n} dz_1dz_2dt_1dt_2  \notag\\
		\leq& 2^{n+1-q}\int_{t_1, t_2\in N} \int_{z_1, z_2\in M}
		\big(|t_1-t_2|+|z_1-z_2|\big)^{q-1-n}dz_1dz_2dt_1dt_2\notag\\
		=& 2^{n+1-q}\int_{t_1, t_2\in N}\int_{z_2\in M}\int_{z'\in M-z_2}
		\big(|t_1-t_2|+|z'|\big)^{q-1-n}dz'dz_2dt_1dt_2\notag\\
		\le& 2^{n+1-q}\int_{t_1, t_2\in N}\int_{z_2\in M}\int_{2B^{n-m}}
		\big(|t_1-t_2|+|z'|\big)^{q-1-n}dz'dz_2dt_1dt_2\notag\\
		=&  \text{\scriptsize $2^{n+1-q}(n-m)\omega_{n-m}\vv_{n-m}(M)$}  \int_{t_1, t_2\in N} \int_0^2
		\big(|t_1-t_2| + r \big)^{q-1-n}r^{n-m-1}dr dt_1dt_2\notag\\
		\leq& \text{\scriptsize $2^{n+1-q}(n-m)\omega_{n-m}\vv_{n-m}(M)$} \int_{t_1, t_2\in N} \int_0^2
		\big(|t_1-t_2|+r\big)^{q-m-2}dr dt_1dt_2\notag\\
		=&\text{\scriptsize $\frac{(n-m)\omega_{n-m}\vv_{n-m}(M) }{2^{q-n-1}(m+1-q)}$}\int_{t_1, t_2\in N}
		\big(|t_1-t_2|^{q-m-1}-\big(|t_1-t_2|+2\big)^{q-m-1}\big)dt_1dt_2\notag\\
		\leq& \text{\scriptsize $\frac{(n-m)\omega_{n-m}\vv_{n-m}(M) }{2^{q-n-1}(m+1-q)}$}\int_{t_1, t_2\in N}
		|t_1-t_2|^{q-m-1}  dt_1dt_2.
		\label{e-Iq3}
	\end{align}
	Similar to the calculation of \eqref{e-Iq2}, we have
	\[ \int_{t_1, t_2\in N} |t_1-t_2|^{q-m-1}dt_1dt_2
	= \frac{2^{q-m+2}}{(q-m)(q-m+1)}a_m^{q-m+1}.\]
	This and \eqref{e-Iq3} give that
	\begin{multline}
		\int_{y_1, y_2 \in J}|y_1-y_2|^{q-1-n} dy_1dy_2 \\
		\leq
		\frac{2^{n-m+3}(n-m)\omega_{n-m}^2}{(m+1-q)(q-m)(q-m+1)} a_m^{q-m+1}a_{m+1}\cdots a_n.
		\label{e-Iq4}
	\end{multline}
	Combining \eqref{e-Iq1},  \eqref{e-Iq2},
	and \eqref{e-Iq4}, we obtain
	\[
	I_{q}(E) \leq  c_{q,m,n} (a_1\cdots a_m)^2a_m^{q-m-1}a_{m+1}\cdots a_n,
	\]
	where
	\begin{equation*}
		c_{q,m,n} = \begin{cases} \dfrac{2^{q-n+2}q(q-1)\omega_{n-1}^2}{(q-n)(q-n+1)n\omega_n}   &m=n \\
			\dfrac{2^{n-m+3}q(q-1)(n-m)\omega_{m-1}^2\omega_{n-m}^2}{(m+1-q)(q-m)(q-m+1)n\omega_n}  &m\le n-1.
		\end{cases}
	\end{equation*}
\end{proof}

\begin{rema}
	\rm
	Note that the estimate in Lemma \ref{e-2} is sharp up to a constant (depending only on $q$ and $n$). To see this
	recall formula \eqref{c6.1},
	\[
	I_{q}(E) = \frac1{n\omega_n} \int_{\sn}\int_{E_u} X_E(x,u)^{q} dx du,
	\]
	where $E_u$ denotes the projection of $E$ onto $u^\perp$.
	For $q>1$, Jensen's inequality gives
	\[
	\frac{1}{\vv(E_u)}\int_{E_u} X_E(x,u)^{q} dx \geq
	\Big(\frac{1}{\vv(E_u)}\int_{E_u} X_E(x,u) dx\Big)^{q}=\Big(\frac{V(E)}{\vv(E_u)}\Big)^{q}.
	\]
	Note that $V(E)=\omega_na_1\cdots a_n$. Since $a_1 \le \cdots \le a_n$, we have  $\vv(E_u) \leq \omega_{n-1}
	 a_2\cdots a_n$.
	Hence,
	\[ I_{q}(E) \geq   \frac1{n\omega_n}  \int_{\sn}V(E)^{q} \vv(E_u)^{-q+1}du
	\geq \frac{\omega_{n-1}^{1-q}}{n\omega_n^{1-q}} (a_2\cdots a_n) a_1^{q}.\]
	When $a_1=\cdots = a_m$, then
	\[
	(a_2\cdots a_n) a_1^{q} =  (a_1\cdots a_m)^2 a_m^{q-m-1} a_{m+1}\cdots a_n,
	\]
	and hence, the ellipsoids $E=E(a_1,\ldots,a_n)$ is such that
	\[
	I_{q}(E) \geq  \frac{\omega_{n-1}^{1-q}}{n\omega_n^{1-q}} (a_1\cdots a_m)^2 a_m^{q-m-1}
	a_{m+1}\cdots a_n.
	\]
\end{rema}

\section{  Solution to the Chord Log-Minkowski Problem}

\vskip5pt

\subsection{Solution to a Maximization Problem} Let $\mu$ be a finite Borel measure on $\sn$.
Define $\Phi_{q,\mu} : \kne \rightarrow \R$ by
\begin{equation}\label{new1}
	\Phi_{q,\mu}(K) = \eop_\mu(K) + \frac1{n+q-1}\log I_{q}(K).
\end{equation}
This functional is continuous with respect to the Hausdorff metric.
We consider the maximization problem,
\[
\sup\{\Phi_{q,\mu}(K) : K\in \kne\}.
\]
If the maximization problem has a solution, then it solves the chord log-Minkowski problem.
The following two lemmas establish the non-degeneracy for a maximizing sequence for the integer case
and for the non-integer case respectively.

We shall make use of the notations {\it outer radius} $R_K$ and {\it inner radius}
$r_K$ of a convex body. For a body $K\in\kne$ these are easily defined by
\[R_K=\max \{|x|: x\in K\}, \qquad r_K=\min \{|x| : x\in K\}.\]

The following inequality will be useful.

\begin{claim}
	If $K\in\kno$  and $1\leq r<s$, then
	\begin{equation}\label{Iq>Im}
		I_r(K) \le
		c(s,r) V(K)^{1-\frac {r-1}{s-1}}I_{s}(K)^{\frac {r-1}{s-1}},
	\end{equation}
	with $c(s,r)=r s^{-\frac{r-1}{s-1}}$.
\end{claim}

To see this note that
from \eqref{c6.3} and
Jensen's inequality,
we have that since $r<s$,
\begin{align}
	\frac{I_{r}(K)}{rV(K)} &=  \frac{1}{V(K)}\int_K \frac{1}{n\omega_n}\int_{\sn}  \rho\lsub{K,z}(u)^{r-1} du dz \notag\\
	&\leq \Big(V(K)n\omega_n\Big)^{-\frac  {r-1}{s-1}}\Big(\int_K\int_{\sn}
	\rho\lsub{K,z}(u)^{s-1} du dz\Big)^{\frac {r-1}{s-1}},\notag
\end{align}
and hence
\begin{equation*}
	I_r(K) \le
	r s^{-\frac{r-1}{s-1}} V(K)^{1-\frac {r-1}{s-1}}I_{s}(K)^{\frac {r-1}{s-1}}.
\end{equation*}

\begin{lemm}\label{m-1}
	Suppose $q\in \{1,\ldots,n\}$, and
	$\mu$ is an even, finite Borel measure on $S^{n-1}$ that satisfies the subspace mass inequality
	\begin{equation*}
		\frac{\mu(\xi_k\cap S^{n-1})}{|\mu|}  <  \frac{k+\min\{k,q-1\}}{n+q-1},
	\end{equation*}
	for all $\xi_k\in G_{n,k}$, and for each $k=1, \ldots, n-1$.
	If $K_l\in\kne$ is a sequence such that $R_{K_l}=1$, for all $l$, and
	$\lim_{l\rightarrow \infty}r_{K_l} = 0,$
	then
	\[\lowlim_{l\rightarrow\infty} \Phi_{q,\mu}(K_l) = -\infty.\]
\end{lemm}

\begin{proof}
	Denote by $E_l$ the John ellipsoid of $K_l$. There is   a sequence of ordered orthonormal bases
	$(e_{1l},..., e_{nl})$ of $\rn$, such that
	\begin{equation}\label{e-4}
		E_l = \Big\{ x \in \rn: \frac{(x\cdot e_{1l})^2}{a_{1,l}^2}
		+ \cdots + \frac{(x\cdot e_{nl})^2}{a_{n,l}^2} \leq 1\Big\},
	\end{equation}
	with $0<a_{1,l}\le\cdots\le a_{n,l}$.
	A subsequence of  $(e_{1l},..., e_{nl})$, which we again denote by $(e_{1l},..., e_{nl})$, will
	converge to an ordered orthonormal basis which we denote by $(e_1,..., e_n)$.
	Since the $K_l$ are origin-symmetric, John's theorem tells us that the John Ellipsoids are such that
	\begin{equation}\label{John}
		E_l\subset K_l\subset \sqrt n E_l.
	\end{equation}
	This inclusion, the facts that for such ellipsoids $R_{E_l}=a_{n,l}$ while $r_{E_l}=a_{1,l}$,
	and the given facts regarding $R_K$ and $r_K$, together yield
	\begin{equation*}
		\text{$1/\sqrt n\leq   a_{n,l} \leq 1$,\qquad and\qquad
			$\lim_{l\rightarrow \infty} a_{1,l} = 0$.}
	\end{equation*}

	From Lemma \ref{e-0}, with $\varepsilon_0 = 1/{\sqrt n}$, we obtain a $t_0>0$ and $c_0\in \R$ so that
	\begin{equation}\label{x13}
		\eop_\mu(E_l) \leq -\frac{2\log (a_{1,l}\cdots a_{q-1, l})}{n+q-1}
		-\frac{\log (a_{q, l} \cdots a_{n,l})}{n+q-1} + t_0\log a_{1,l} + c_0.
	\end{equation}
	Now \eqref{new1} and both parts of \eqref{John}, the monotonicity of $\eop_\mu$ and of $I_q$ and \eqref{homog},
	together with \eqref{x13} give
	\begin{align}\label{e-3}
		\Phi_{q,\mu}(K_l)
		&\leq
		\eop_\mu(E_l) + \frac1{n+q-1}\log I_{q}(\sqrt{n}E_l)\nonumber\\
		&=
		\eop_\mu(E_l) + \frac1{n+q-1}\log I_{q}(E_l) + \frac{\log n}{2}\nonumber\\
		&\le
		-\frac{2\log (a_{1,l}\cdots a_{q-1, l})}{n+q-1}
		-\frac{\log (a_{q, l} \cdots a_{n,l})}{n+q-1}
		+ \frac{\log I_{q}(E_l)}{n+q-1}\notag\\
		&\hspace{130pt}+ t_0\log a_{1,l} + c_0 + \frac{\log n}2.
	\end{align}

	Choose a $q'\in(q,q+1)$ sufficiently close to $q$ so that
	\[
	t_0 +\dfrac {q-1}{n+q-1}\big(\dfrac {q-1}{q'-1}-1\big)>0.
	\]

	From \eqref{Iq>Im} with $q=r$ and $q'=s$, the facts that $V(E_l) =  \omega_n a_{1,l}\cdots a_{n,l}$
	with $0<a_{1,l}\le\cdots\le a_{n,l} \le 1$, and Lemma \ref{e-2}
	with the integer $q$ taking the place of $m$ and $q'$ the place of $q$, we get
	\begin{align*}
		&\log I_{q}(E_l)\\
		\leq&
		\log c(q',q) + \log V(E_l)^{1-\frac {q-1}{q'-1}}
		+ \log I_{q'}(E_l)^{\frac {q-1}{q'-1}} \\
		\leq&
		\big(1 - \tfrac {q-1}{q'-1}\big)\log (\omega_n a_{1,l}\cdots a_{n,l}) +
		\big(\tfrac {q-1}{q'-1}\big)\log\big(c_{q',q,n}  (a_{1,l}\cdots a_{q,l})^2 a_{q,l}^{q'-q-1}
		a_{q+1,l}\cdots a_{n,l} \big) \\
		&\hspace{125pt} +  \log c(q',q) \\
		\leq&
		\big(1+\tfrac {q-1}{q'-1}\big)\log (a_{1,l}\cdots a_{q-1,l}) +
		\log(a_{q,l}\cdots a_{n,l}) +  \log c'(q',q,n),
	\end{align*}
	where $c'(q',q,n) = c(q',q)\omega_n^{1-(q-1)/(q'-1)}c_{q',q,n}^{(q-1)/(q'-1)}$
	and $c_{q',q,n}$ is  given by \eqref{cqmp}.
	From this together with \eqref{e-3}, the facts that $\frac {q-1}{q'-1}-1<0$
	and $a_{1,l}\leq...\leq a_{q-1,l}$, and our choice of $q'$, we have
	\begin{align*}
		&\Phi_{q,\mu}(K_l)\\
	\leq	& \frac1{n+q-1}\Big(\frac {q-1}{q'-1}-1\Big)\log(a_{1,l}\cdots a_{q-1,l})
		+ t_0\log a_{1l} + c_0 + \frac{\log n}2  +  \frac{\log c'(q',q,n)}{n+q-1} \\
	\leq	& \Big(t_0+\frac {q-1}{n+q-1}\big(\frac {q-1}{q'-1}-1\big)\Big)
		\log a_{1l} + c_0 + \frac{\log n}2
		+  \frac{\log c'(q',q,n)}{n+q-1} \\
		\rightarrow& -\infty,
	\end{align*}
	as $l\to\infty$.
\end{proof}

We now turn to the case where $q$ is not an integer.

\begin{lemm} \label{m-2}
	Suppose $q\in (1,n+1)$ is not an integer, and
	$\mu$ is an even finite Borel measure on $S^{n-1}$ that satisfies the subspace mass inequality
	\eqref{mass1+}. If $\{K_l\}$ is sequence of convex bodies in $\kne$ such that $R_{K_l}=1$ and
	$\lim_{l\rightarrow \infty}r_{K_l} = 0,$
	then
	\[
	\lowlim_{l\rightarrow\infty} \Phi_{q,\mu}(K_l) = -\infty.
	\]
\end{lemm}

\begin{proof} Let $m=\lfloor q\rfloor$. As in the proof of Lemma \ref{m-1},
	denote by $E_l$ the John ellipsoid of $K_l$, which is given by \eqref{e-4} and satisfies
	inclusions \eqref{John}. As before, it easily follows that
	$(1/\sqrt n)\leq   a_{n,l} \leq 1$, and $\lim_{l\rightarrow\infty}a_{1,l}=0$.
	
	By Lemma \ref{e-0}, for $\varepsilon_0 =  1/{\sqrt n}$, there exist $t_0,l_0>0$ and $c_0\in \R$ so that
	\begin{align}
		\eop_\mu(E_l) \leq -\frac{2\log(a_{1,l}\cdots a_{m-1,l})}{n+q-1}
		- \frac{q-m+1}{n+q-1}\log a_{m,l}\notag\\
		 -
		\frac{\log(a_{m+1,l}\cdots a_{n,l})}{n+q-1} + t_0\log a_{1,l} + c_0, \label{e-5}
	\end{align}
	for all $l>l_0$.
	From the definition of $\Phi_{q,\mu}$, the
	monotonicity of $\eop_\mu$ and $I_q$,
	the inclusion \eqref{John} and the fact that $I_q$ is $(n+q-1)$-homogeneous, we get
	\begin{equation}\label{e-3+}
		\Phi_{q,\mu}(K_l) \leq  \eop_\mu(E_l) + \frac1{n+q-1} \log I_q(E_l)+ \frac12\log n.
	\end{equation}
	The estimate in Lemma \ref{e-2} says
	\[
	I_{q}(E_l)\leq c_{q,m,n}  (a_{1,l}\cdots a_{m,l})^2a_{m,l}^{q-m-1}  a_{m+1,l}\cdots a_{n,l}.
	\]
	This together with
	\eqref{e-5} and \eqref{e-3+} gives
	\[
	\Phi_{q,\mu}(K_l) \leq t_0\log a_{1,l} + c_0 + \frac1{n+q-1}\log c_{q,m,n}
	+ \frac12\log n \quad \rightarrow -\infty,
	\]
	as $l\to\infty$.
\end{proof}

\begin{theo}\label{s-max} Let $q\in [1,n+1)$.
	If $\mu$ is an even finite Borel measure on $S^{n-1}$ that satisfies
	the subspace mass inequality \eqref{mass1+}, then there exists $K_0\in \kne$ such that
	\[\Phi_{q,\mu}(K_0) = \sup\{\Phi_{q,\mu}(K) : K\in \kne \}.\]
\end{theo}

\begin{proof}
	Suppose that $K_l\in \kne$ is a maximizing sequence for $\Phi_{q,\mu}$.
	Since
	\[
	\Phi_{q,\mu}(B^n) = \frac1{n+q-1} \log I_{q}(B^n),
	\]
	we know that $\Phi_{q,\mu}({K_l})$ is bounded from below by a constant.
	Since the functional $\Phi_{q,\mu}$ is $0$-homogeneous, we may assume that all the $R_{K_l}=1$.
	
	We claim that $K_l$ has a subsequence for which the inner radii
	are bounded from blow
	by a positive constant. If not then for a subsequence, which we again call $K_l$, we have
	$\lim_{l\rightarrow \infty}r_{K_l} =0$.
	By Lemmas \ref{m-1} and \ref{m-2}, whether $q$ is an integer or not,
	there exists a subsequence of  $K_l$, say $K_{l_i}$, for which
	\[
	\Phi_{q,\mu}(K_{l_i})   \rightarrow -\infty,
	\]
	which would  contradict the fact that $\Phi_{q,\mu}(K_{l_i})$ is bounded from below.
	
	Thus, we may assume all the $R_{K_{l}}=1$ and $r_{K_{l}}$ is bounded from below by a positive constant. This implies that $\{K_{l}\}$ has a subsequence, which we again denote by $K_l$
	that converges to a convex body
	$K_0\in \kne.$  Since $\{K_{l}\}$ is a maximizing sequence for
	the continuous functional $\Phi_{q,\mu}$, the body $K_0$ must be a maximizer.
\end{proof}

\subsection{Solution to the Chord Log-Minkowski Problem}

The following theorem establishes the existence of solution to the chord log-Minkowski problem
for origin-symmetric convex bodies.

\begin{theo}
	Suppose real $q\in [1,n+1]$.
	If $\mu$ is an even finite Borel measure on $S^{n-1}$ satisfying the subspace mass inequality
	\begin{equation*}
		\frac{\mu(\xi_k\cap S^{n-1})}{|\mu|}  <  \frac{k+\min\{k,q-1\}}{n+q-1},
	\end{equation*}
	for each $\xi_k\in G_{n,k}$, and for each $k=1, \ldots, n-1$,
	then there exists $K\in \kne$ such that
	\[
	G_q(K,\cdot) = \mu.
	\]
\end{theo}
\begin{proof}
	We may assume that $q\in [1,n+1)$
	by observing that solving the even
	chord log-Minkowski problem for the case where $q=1$ immediately provides a solution to the
	case $q=n+1$ since from \eqref{c12} and \eqref{c13} we know that
	\[
	G_{n+1}(K,\cdot) = \frac{n+1}{\omega_n} V(K) G_1(K,\cdot).
	\]
	Note that the subspace concentration conditions are identical for $q=1$ and $q=n+1$.

	Theorem \ref{s-max} demonstrates the existence of a convex body  $K_0\in \kne$ that solves the maximization problem
	\[\Phi_{q,\mu}(K_0) = \sup\{\Phi_{q,\mu}(K) : K\in \kne \}.\]
	Let $g\in C(S^{n-1})$ be an even function, and $h_t= h_{K_0}e^{tg}$. Define the Wulff shape $K_t$ by
	\[
	K_t =\{x \in \rn : x\cdot v \le h_t(v) \text{ for all } v\in \sn\}.
	\]
	Obviously, for sufficiently small $t$, the function $h_t$ is even, positive, and continuous
	on $S^{n-1}$. Thus, each $K_t\in\kne$.
	Thus, the function $t\mapsto \Phi_{q,\mu}(K_t)$ attains its maximum at $t=0$.
	However, $t\mapsto \Phi_{q,\mu}(K_t)$ may not be differentiable at $t=0$.
	Define
	\[
	\phi(t) = -\frac1{|\mu|}\int_{\sn} \log h_t \, d\mu + \frac1{n+q-1}I_q(K_t).
	\]
	We combine the fact that the function $t\mapsto \int_{\sn} \log h_t \, d\mu$ is
	trivially differentiable with Theorem \ref{V3+} to conclude that $\phi$
	is differentiable at $0$.
	Since trivially $h_{K_t} \le h_t$ we have
	\[
	\phi(t) \leq \Phi_{q,\mu}(K_t) \leq \Phi_{q,\mu}(K_0) = \phi(0).
	\]
	Thus, $\phi$ has a local maximum at $0$, and hence,
	\[
	\frac d{dt}\Big|_{t=0} \phi(t)=0.
	\]
	Theorem \ref{V3+} states that  $t\mapsto I_q(K_t)$ is differentiable at $t=0$, and
	\begin{equation}
		\frac d{dt}\Big|_{t=0}I_q(K_t) = (n+q-1)\int_{\sn}g(v) dG_q(K_0,v).
	\end{equation}
	It follows that
	\[
	-\frac1{|\mu|}\int_{\sn} g(v) d\mu(v) + \int_{\sn}g(v) dG_q(K_0,v)=0.
	\]
	But $g$ was an arbitrary even, continuous function on $\sn$, and thus
	\[
	G_q(K_0, \cdot) = \frac{\mu}{|\mu|}.
	\]
	Defining $K\in\kne$ by
	\[
	K =  |\mu|^{\frac1{n+q-1}}K_0
	\]
	results in
	\[
	G_q(K,\cdot) = |\mu| G_q(K_0,\cdot) = \mu,\]
	by recalling the $(n+q-1)$-homogeneity of the cone chord measure.
\end{proof}

\section{  Concentration of cone chord measures}

When a geometric measure of a convex body has a density, it has no concentration in any subspaces.
The surface area measure can have almost all of its mass concentrated in a single subspace.
For example, a right cylinder with a large radius but small
height, has its surface area almost exclusively concentrated on its bases; i.e.,
its surface area measure is concentrated on the one-dimensional space normal to its bases.
The \lq\lq concentration in subspaces\rq\rq\ question for the surface area measure has a trivial answer:
There are no subspace restrictions.

Non-trivial concentration properties of geometric measures encode
important geometric properties of convex bodies.
Non-trivial concentration properties of geometric measures were first studied for cone volume measures
of polytopes. He-Leng-Li \cite{HLL} and Henk-Sch\"urmann-Wills \cite{HSW} independently discovered
the concentration property of cone volume measure for symmetric polytopes in $\rn$.
Then non-trivial concentration properties
for polytopes that are not necessarily symmetric were established by Xiong \cite{X1}
in $\mathbb R^3$ and by
Henk-Linke \cite{HL14adv} in $\rn$ for all $n$. For symmetric convex bodies,
the necessity of the concentration criteria were demonstrated in \cite{BLYZ13jams}. In \cite{BLYZ13jams}
the concentration criteria (called the ``subspace concentration
condition") were shown to be sufficient for establishing the existence of a solution to the log-Minkowski
problem for cone volume measure in $\rn$. This was established earlier by Stancu \cite{Sta1} for
polygons in $\mathbb R^2$.
Thus, the concentration property of cone volume measure is a characterization property.
Dual curvature measures also have non-trivial concentration properties. Sharp estimates
for subspace concentrations
for dual curvature measures of symmetric convex bodies were obtained
by B\"or\"oczky-Henk-Pollehn \cite{BHP17jdg} and
Henk-Pollehn \cite{HP}. These concentration properties were shown to be the sufficient conditions
for the existence
of solutions to the dual Minkowski problem, see \cite{BLYZZ1, Zhao17jdg}.
Hering-Nill-S\"uss  \cite{HNS19arxiv} used
the recent theorem of Chen-Donaldson-Sun
\cite{CDS15jamsI, CDS15jamsII, CDS15jamsIII} on the existence of a K\"ahler-Einstein metric,
on  a $K$-stable Fano manifold, to show that the subspace concentration condition for smooth
reflexive polytopes is implied by the $K$-stability property of smooth toric Fano varieties.

In this section, we establish a subspace mass inequality satisfied by cone chord measures. This
implies that in some cases the subspace mass inequality \eqref{mass1+} is also
a necessary condition for the existence of solutions to the chord log-Minkowski problem.

For $K\in \kn$ and $z\in \rn$,  let
\[
f_{\xi_i}(z) = \vv_i\big(K\cap (z+\xi_i) \big), \ \ \xi_i \in G_{n,i}, \ \ i=1, 2, \ldots, n,
\]
and $f_{\xi_0}=\mathbf{1}_K$.

Since $K$ is convex, it follows that for $z_1,z_2\in K$ and $t\in [0,1]$,
\[
K \cap \big((1-t)z_1 +tz_2  + \xi_i\big) \supset
(1-t)\big(K \cap (z_1+\xi_i)\big)+t \big(K \cap (z_2+\xi_i)\big).
\]
This fact, when combined with the Brunn-Minkowski inequality, shows that
the function $f_{\xi_i}^{1/i}$ is concave on $K$.
Thus, $f_{\xi_i}$ is unimodal\footnote{Recall that a function $f : \rn \to [0,\infty)$ is
	said to be {\it unimodal} if the level set $\{x\in \rn : f(x) \ge t \}$ is convex for each $t >0$.}.
When $K$ is origin-symmetric,  $f_{\xi_i}$
is an even unimodal function.

\begin{lemm}\label{cm1x}
	If  $K\in\kno$ and $\eta\subset\sn$ is Borel, then for $i=1,2, \ldots, n$,
	\begin{equation}
		I_{i+1}(K)  =\frac {i+1}{\omega_i} \int_{G_{n,i}} \int_{K}
		\vv_i\big(K\cap (z+\xi_i) \big) \, dz d\xi_i, \label{eq-sub1}\end{equation}
		\begin{equation}G_{i+1}(K,\eta)  = \frac{2(i+1)}{(n+i)\omega_i}\int_{G_{n,i}} \int_{{\nu_K^{-1}}(\eta)}
		(z\cdot \nu_K(z)) \vv_i\big(K\cap (z+\xi_i) \big)  \,  d\hm(z)d\xi_i. \label{eq-sub2}
	\end{equation}
\end{lemm}

\begin{proof}
	Let $q=i+1$. From \eqref{c2.3}, \eqref{c6.3}, and Fubini's theorem, we have
	\begin{align*}
		I_q(K) &=  \frac q{\omega_n} \int_K \wt V_{q-1}(K, z)\, dz  \\
		&= \frac{i+1}{\omega_n} \int_{K} \Big(\frac{\omega_n}{\omega_i} \int_{G_{n,i}}
		f_{\xi_i}(z) \, d\xi_i\Big) dz \nonumber \\
		&= \frac {i+1}{\omega_i} \int_{G_{n,i}} \int_{K} f_{\xi_i}(z) \, dz d\xi_i.
	\end{align*}
	From \eqref{c11.1} and \eqref{c2.3}, we have
	\begin{align*}
		G_q(K,\eta) &=  \frac{2q}{(n+q-1)\omega_n}
		\int_{{\nu_K^{-1}}(\eta)} (z\cdot \nu_K(z)) \wt V_{q-1}(K,z)\, d\hm(z) \\
		&= \frac{2(i+1)}{(n+i)\omega_i}\int_{G_{n,i}} \int_{{\nu_K^{-1}}(\eta)}
		(z\cdot \nu_K(z)) f_{\xi_i}(z)\,  d\hm(z)d\xi_i.
	\end{align*}
\end{proof}

We use a technique developed in B\"or\"oczky-Henk-Pollehn \cite{BHP17jdg} which
is the following lemma.

\begin{lemm}\label{bmi}
	Suppose $L\in\cn$ is such that $L\subset\xi_i\in  G_{n,i}^a$ and
	$f:\rn\to\rn$ is even and unimodal.  Then for for all $t\in (0,1)$,
	\begin{equation}\int_{(1-t)L+t (-L)} f(z)\, d\thh^{i}(z)
		\geq \int_{L} f(z)\, d\thh^{i}(z).\end{equation}
\end{lemm}

We will use the following formula for integration by generalized polar coordinates:
If $K\in \kno$ and $\phi : \partial K \to \R$ is integrable, then
\begin{equation}\label{gsc}
	\int_K \phi(\rhok(z) z)\, dz = \frac1n \int_{\partial K} \phi(z)
	(z \cdot \nu_K(z)) \, d\hm(z).
\end{equation}

\begin{theo}\label{mc}
	If $K\in\kne$ and
	$q\in \{1,\dots,n+1\}$, then
	the cone chord measure $G_q(K,\cdot)$ satisfies the following measure concentration inequalities,
	\begin{numcases}{G_q(K,\xi_k\cap S^{n-1}) \leq}
		\min\big\{\dfrac{2k}{n+q-1}, 1\big\}
		G_q(K,S^{n-1}),  \ \   &q = 2,3,...,n+1,
		\label{smq}
		\\
		\dfrac{k }{n} G_1(K,S^{n-1}), &q=1 \notag
	\end{numcases}
	\ for each  $\xi_k\in G_{n,k}$ and each $k=1,\ldots,n-1$.
\end{theo}

\begin{proof}
	Suppose $\xi\in G_{n,k}$. For $z\in \rn$, let $z=(x,y)$ with $x=P_\xi z$ and $y=P_{\xi^\perp}z$.
	The set $\nu_K^{-1}(\xi\cap \sn) \subset \partial K$ is on the surface of
	the generalized cylinder that is the direct product of $\xi^\perp$ and
	$K|\xi$. The surface area element of the generalized cylinder is a direct
	product of Lebesgue measure in $\xi^\perp$ and the surface area element
	of $K|\xi$. Thus, for $z=(x,y)\in \nu_K^{-1}(\xi\cap \sn)$, we have
	$x\in \partial (K|\xi)$ and $y\in x+\xi^\perp$, while
	\begin{equation}\label{mc1}
		d\hm(z) = dy\, d\thh^{k-1}(x).
	\end{equation}
	When $\nu_K(z)$ and $\nu_{K|\xi}(x)$ are defined, we have
	\begin{equation}\label{mc2}
		z\cdot \nu_K(z) = x\cdot \nu_{K|\xi}(x).
	\end{equation}
	
	Therefore, from \eqref{mc1}, \eqref{mc2} and \eqref{gsc}, we have for $i=0, 1, \ldots, n$,
	\begin{align}
		&\int_{\nu_K^{-1}(\xi\cap S^{n-1})} (z\cdot \nu_K(z)) f_{\xi_i}(z)\, d\hm(z) \notag\\
		=& \int_{\partial(K|\xi)} \int_{(x+\xi^\perp) \cap \partial K} (x\cdot \nu_{K|\xi}(x)) f_{\xi_i}(z)\,
		dy\, d\thh^{k-1}(x) \notag\\
		=& \int_{\partial(K|\xi)} \phi(x) (x\cdot \nu_{K|\xi}(x)) d\thh^{k-1}(x) \notag \\
		=& k \int_{K|\xi}  \phi(\rho\lsub{K|\xi}(x)x) \, dx,
		\label{eq-coneH}
	\end{align}
	where $\partial (K|\xi) \ni x\mapsto\phi(x) = \int_{(x +\xi^\perp) \cap \partial K} f_{\xi_i}(z) dy$
	for $z=(x,y)\in \partial K$.

	Clearly,
	\begin{equation}\label{cm3}
		\int_K f_{\xi_i}(z)dz = \int_{K|\xi}\int_{(x+\xi^\perp) \cap K} f_{\xi_i}(z) \, dy\,dx.
	\end{equation}
	Let $\tilde x = \rho\lsub{K|\xi}(x)x$. Since $K|\xi$ is origin-symmetric, $\tilde x$ and $-\tilde x$
	are antipodal points on $\partial (K|\xi)$,
	and hence
	\[(-\tilde x + \xi^\perp)\cap K = - ((\tilde x + \xi^\perp)\cap K).\]
	Obviously $\tilde x = \rho\lsub{K|\xi}(x)x \in
	\partial K$. Thus,
	let $Q=(\tilde x + \xi^\perp)\cap K=(\tilde x +\xi^\perp) \cap \partial K$.
	There exists $t\in[0,1]$ such that $x=(1-t)\tilde x + t (-\tilde x)$.
	Then, since $K$ is convex we have,
	\[
	(x+\xi^\perp) \cap K \supset (1-t)(\tilde x + \xi^\perp)\cap K
	+ t (-\tilde x + \xi^\perp)\cap K = (1-t)Q + t(-Q).
	\]
	Since $f_{\xi_i} (z)$ is an even unimodal function, from Lemma \ref{bmi}, we have
	\begin{equation}\label{cm4}
		\int_{(x+\xi^\perp) \cap K} f_{\xi_i}(z) \,  dy \geq
		\int_{(1-t)Q + t(-Q)}  f_{\xi_i}(z)\,  dy \geq  \int_{Q}  f_{\xi_i}(z)\,  dy.
	\end{equation}
	By \eqref{cm3} and \eqref{cm4},
	\[
	\int_K f_{\xi_i}(z)dz \ge \int_{K|\xi} \int_{(\tilde x +\xi^\perp) \cap \partial K}
	f_{\xi_i}(z)\, dy\,dx = \int_{K|\xi} \phi(\tilde x)\, dx.
	\]
	This and \eqref{eq-coneH} give
	\begin{equation}\label{cm5}
		\int_K f_{\xi_i}(z)dz \ge
		\frac1k \int_{\nu_K^{-1}(\xi\cap S^{n-1})} (z\cdot \nu_K(z)) f_{\xi_i}(z)\, d\hm(z).
	\end{equation}
	Integrating both sides of \eqref{cm5} and using \eqref{eq-sub1} and \eqref{eq-sub2}, gives us
	
	\begin{equation*}
		I_{i+1}(K) \geq
		\begin{cases}
			\dfrac{n+i}{2k} G_{i+1}(K, \xi \cap S^{n-1}), \ \   &i=1, \ldots, n,
			\\
			\phantom{\cdot}  &\phantom{\cdot}
			\\
			\dfrac nk G_1(K, \xi\cap \sn), &i=0.
		\end{cases}
	\end{equation*}
	
\end{proof}

The following proposition shows that the subspace mass inequality \eqref{smq}
is sharp when $k<q-1$.

\begin{prop} Let $k$ and $q$ be integers such that $1\le k< q-1\le n$.
	For each $\xi\in G_{n,k}$, there exists
	a sequence of convex bodies  $K_j\in \kno$, such that
	\begin{equation}\label{eq-sharp}
		\lim_{j\to \infty} \frac{G_q(K_j, \xi\cap S^{n-1})}{G_q(K_j, S^{n-1})}
		= \frac{2k}{n+q-1}.
	\end{equation}
\end{prop}

\begin{proof} Let $D\subset \xi$ be a $k$-dimensional convex body in $\xi$, and let $Q\subset \xi^\perp$ be
	an $(n-k)$-dimensional convex body in $\xi^\perp$. Suppose that both $D$ and $Q$ contain
	the origin in their relative interiors. Let $K=  D\times Q \in \kno$.  We proceed using
	the notation in the proof
	of Theorem \ref{mc}.  Note that $K|\xi = D$ and
	$(x+\xi^\perp) \cap \partial K = x +Q$ for $x\in \partial D$.
	Similar to the calculation for \eqref{eq-coneH}, we have
	\begin{align}
		&G_q(K, \xi\cap S^{n-1})\notag\\
		=& \frac {2q}{(n+q-1)\omega_n} \int_{\nu_K^{-1}(\xi\cap \sn)}
		(z\cdot \nu_K(z)) \wt V_{q-1}(K,z)\, d\hm(z) \notag \\
		 =&\frac {2q}{(n+q-1)\omega_n} \int_{\partial(K|\xi)} \int_{(x+\xi^\perp)\cap \partial K}
		(x\cdot \nu_{K|\xi}(x)) \wt V_{q-1}(K,z)\, dy\, d\thh^{k-1}(x) \notag \\
		=&\frac {2q}{(n+q-1)\omega_n} \int_{\partial D} x \cdot \nu_D(x) \phi(x) d\thh^{k-1}(x) \notag\\
		=& \frac {2qk}{(n+q-1)\omega_n} \int_D \phi(\tilde x) dx,
		\label{coest1}
	\end{align}
	where $\phi(x)=\int_{Q} \wt V_{q-1}(K,(x,y))\, dy$, for
	$x\in \partial D$, and $y\in Q$, and
	$\tilde x = \rho\lsub{D}(x)x \in \partial D$ for $x\in D$. For $z\in K$,
	let $\tilde z = (\tilde x, y)$. Then by \eqref{c2+},
	\begin{align}\label{co1.1}
		\int_D \phi(\tilde x) dx &= \frac{q-1}n \int_{x\in D} \int_{y\in Q} \int_{w\in K}
		|w-\tilde z|^{q-1-n} dwdydx\\
		&=\frac{q-1}n\int_{w, z \in K} |w-\tilde z|^{q-1-n} dwdz.
	\end{align}
	
	Let $K_j = \big(\frac1j D\big)\times Q$. Then by \eqref{coest1} and \eqref{co1.1},
	\begin{align}
		G_q(K_j, \xi\cap S^{n-1}) &= \frac {2q(q-1)k}{(n+q-1)n\omega_n}
		\int_{w, z \in K_j} |w-\tilde z|^{q-1-n}\, dw\,dz
		\notag \\
		&= \frac {2q(q-1)k}{(n+q-1)n\omega_n j^{2k}} \int_{w, z \in K}
		\big|\big(\tfrac{w_1-\tilde x}j, w_2-y\big)\big|^{q-1-n}\, dw\,dz, \label{co1.2}
	\end{align}
	where $w=(w_1, w_2) \in K$, with $w_1\in D$, and $w_2\in Q$.
	
	Observe that from \eqref{c6.3} and \eqref{c2+}, we have
	\begin{align}
		I_q(K_j)
		&= \frac{q}{\omega_n}\int_{K_j} \wt V_{q-1}(K_j,z)dz\\
		&=  \frac{q(q-1)}{n\omega_n} \int_{w, z \in K_j} |w-z|^{q-1-n}\, dw\,dz \notag \\
		&= \frac{q(q-1)}{n\omega_nj^{2k}}\int_{w, z \in K}
		\big|\big(\tfrac{w_1-x}j, w_2-y\big)\big|^{q-1-n} dwdz.\label{coest2}
	\end{align}
	
	Since both sequences of functions in the integrals in \eqref{co1.2} and \eqref{coest2}
	are increasing with respect to $j$, by the monotone convergence theorem,
	both integrals converge to
	\[
	\int_{w, z \in K} |w_2-y|^{q-1-n}\, dw\,dz
	\]
	which is a positive number when $q-1-n > k-n$ with $\dim Q =  n-k$.
	This together with \eqref{co1.2} and \eqref{coest2} give \eqref{eq-sharp}.
\end{proof}

Note that when $k<q-1$, the mass inequality \eqref{mass1+} is the same inequality as
\eqref{smq}, and when $k\ge q-1$, inequality \eqref{mass1+} is stronger than the
inequality given by \eqref{smq}.







\thanks{The authors thank Xiaxing Cai for comments on a draft of this work. Dongmeng Xi thanks the Courant Institute for its warm hospitality during his extended visit. Research of Lutwak, Yang, and Zhang was supported, in part, by
  NSF Grant  DMS--2005875. Research of Xi was supported, in part, by NSFC Grant 12071277 and  STCSM Grant 20JC1412600.}



\frenchspacing
\bibliographystyle{cpam}

\end{document}